\documentclass
{amsart}

\usepackage{amsmath}
\usepackage{amsfonts}
\usepackage{amssymb}
\usepackage{latexsym}
\newtheorem{theorem}{Theorem}[section]

\newtheorem{proposition}[theorem]{Proposition}
\newtheorem{corollary}[theorem]{Corollary}

\newtheorem{remark}[theorem]{Remark}
\numberwithin{equation}{section}

\def\sD{{\mathfrak D}}      
   \def\sH{{\mathfrak H}}   
   \def\sK{{\mathfrak K}}   \def\sL{{\mathfrak L}}
\def\sM{{\mathfrak M}}   \def\sN{{\mathfrak N}}

      \def\dC{{\mathbb C}}
\def\dD{{\mathbb D}}

   \def\dN{{\mathbb N}}   
      
   \def\dT{{\mathbb T}}

\def\cA{{\mathcal A}}   \def\cB{{\mathcal B}}   \def\cC{{\mathcal C}}
\def\cD{{\mathcal D}}   \def\cE{{\mathcal E}}   \def\cF{{\mathcal F}}
\def\cG{{\mathcal G}}   \def\cH{{\mathcal H}}   
   \def\cK{{\mathcal K}}   \def\cL{{\mathcal L}}
\def\cM{{\mathcal M}}      
   \def\cQ{{\mathcal Q}}   
\def\cS{{\mathcal S}}   \def\cT{{\mathcal T}}   \def\cU{{\mathcal U}}
\def\cV{{\mathcal V}}   \def\cW{{\mathcal W}}   
   \def\cZ{{\mathcal Z}}

\def\bJ{{\mathbf J}}
\def\bL{{\mathbf L}}
\def\bS{{\mathbf S}}

\def\wt{\widetilde}
\def\wh{\widehat}

\def\f{\varphi}
\def\uphar{{\upharpoonright\,}}

\def\ran{{\rm ran\,}}

\def\cran{{\rm \overline{ran}\,}}

\begin{document}
\title
 {The Schur problem and block operator CMV matrices}

\author{
Yury~Arlinski\u{i}}
\address{Department of Mathematical Analysis \\
East Ukrainian National University \\
Kvartal Molodyozhny 20-A \\
Lugansk 91034 \\
Ukraine} \email{yury.arlinskii@gmail.com}


\thispagestyle{empty}

\begin{abstract}
The CMV matrices and their sub-matrices are applied to the
description of all solutions to the Schur interpolation problem for
contractive analytic operator-valued functions in the unit disk (the
Schur class functions).
\end{abstract}
\maketitle
\tableofcontents

\section{Introduction}
In what follows the class of all continuous linear operators defined
on a complex Hilbert space $\sH_1$ and taking values in a complex
Hilbert space $\sH_2$ is denoted by $\bL(\sH_1,\sH_2)$ and
${\bL}(\sH):= {\bL}(\sH,\sH)$. All infinite dimensional Hilbert
spaces are supposed to be separable. We denote by $I_H$ the identity
operator in a Hilbert space  $H$ and omit the symbol $H$ in these
notations if there is no danger of confusion; by $P_\cL$ the
orthogonal projection onto the subspace (the closed linear manifold)
$\cL$. The notation $T\uphar \cL$ means the restriction of a linear
operator $T$ on the set $\cL$. The range and the null-space of a
linear operator $T$ are denoted by $\ran T$ and $\ker T$,
respectively. We use the standart symbols $\dC$, $\dN$, and $\dN_0$
for the sets of complex numbers, positive integers, and nonnegative
integers, respectively. An operator $T\in\bL(\sH_1,\sH_2)$ is said
to be
\begin{itemize}
\item \textit{contractive} if $\|T\|\le 1$;

\item \textit{isometric} if $\|Tf\|=\|f\|$ for all $f\in \sH_1$
$\iff T^*T=I_{\sH_1}$;

\item \textit{co-isometric} if $T^*$ is isometric $\iff
TT^*=I_{\sH_2}$;
\item \textit{unitary} if it is both isometric and co-isometric.
\end{itemize}
Given a contraction $T\in \bL(\sH_1,\sH_2)$, the operators $
D_T:=(I_{\sH_1}-T^*T)^{1/2}$ and $D_{T^*}:=(I_{\sH_2}-TT^*)^{1/2} $
are called the \textit{defect operators} of $T$, and the subspaces
$\sD_T=\cran D_T,$ $\sD_{T^*}=\cran D_{T^*}$ the \textit{defect
subspaces} of $T$. 
The defect operators satisfy the relations $ TD_{T}=D_{T^*}T,$
$T^*D_{T^*}=D_{T}T^*.$
Let $\sM$ and $\sN$ be Hilbert spaces. The \textit{Schur class}
${\bf S}(\sM,\sN)$ is the set of all functions $\Theta(z)$  with
values in $\bL(\sM,\sN)$, holomorphic in the unit disk
$$
\dD=\{z\in\dC:|z|<1\}$$
 and such that
$\|\Theta(z)\|\le 1$ for all $z\in\dD$. Let $\Theta$ be holomorphic
in $\dD$ operator valued function acting between Hilbert spaces
$\sM$ and $\sN$ and let
\[
\Theta(z)=\sum\limits_{n=0}^{\infty}z^n C_n,\; z\in\dD,\;
C_n\in\bL(\sM,\sN),\; n\ge 0
\]
be the Taylor expansion of $\Theta$. Consider the lower triangular
(analytic) Toeplitz matrix
\[
T_\Theta:=\begin{bmatrix}C_0&0&0&0&\ldots&\ldots\cr
C_1&C_0&0&0&\ldots&\ldots\cr C_2&C_1&C_0&0&0&\ldots\cr
C_3&C_2&C_1&C_0&0&\ldots\cr
\vdots&\vdots&\vdots&\vdots&\vdots&\vdots
\end{bmatrix}.
\]
Let $\cH$ be a separable Hilbert space and let $\ell_2(\cH)$ be a
Hilbert space of all vectors $\vec f=[f_0, f_1,\ldots]^T,$ $f_k
\in\cH,$ $k\in\dN_0$,
$\sum\limits_{k=0}^\infty ||f_k||^2_\cH<\infty$, with the
inner product $(\vec f,\vec g)=\sum\limits_{k=0}^\infty (f_k,
g_k)_\cH.$ As is well known \cite{BC, FoFr}
\[
\Theta\in\bS(\sM,\sN)\iff
T_\Theta\in\bL\left(\ell_2(\sM),\ell_2(\sN)\right)\quad\mbox{is a
contraction}.
\]
Set for $n\in\dN_0$
\[\begin{array}{l}
\sM^{n+1}=\underbrace{\sM\oplus\sM\oplus\cdots\oplus\sM}_{n+1},\; 
\sN^{n+1}=\underbrace{\sN\oplus\sN\oplus\cdots\oplus\sN}_{n+1}.\\
\end{array}
\]
Clearly, if $T_\Theta$ is a contraction, then the operator
$T_{\Theta,n}\in \bL\left(\sM^{n+1},\sN^{n+1}\right)$ given by the
block operator matrix
\[
T_{\Theta,n}:=\begin{bmatrix}C_0&0&0&\ldots&0\cr
C_1&C_0&0&\ldots&0\cr \vdots&\vdots&\vdots&\vdots&\vdots\cr
C_n&C_{n-1}&C_{n-2}&\ldots&C_0\end{bmatrix}
\]
is a contraction  for each $n$.


The following interpolation problem is called the \textit{Schur
problem}:
\begin{equation}
\label{SCHPRB}
\begin{array}{l}
 \mbox{Let}\;\sM\;\mbox{and}\;\sN\;\mbox{be Hilbert
spaces}.\\
\mbox{Given operators}\; C_k\in\bL(\sM,\sN),\; k=0,1,\ldots,
N,\\
\mbox{does there exist}\; \Theta\in\bS(\sM,\sN)\;\mbox{such that}\;
C_0,C_1,\ldots, C_N\;\mbox{are the first}\\
 \mbox{Taylor
coefficients of}\; \Theta,\, i.e.,\; \cfrac{\Theta^{(k)}(0)}{k!}=C_k
\quad\mbox{for}\quad k=0,1,\ldots,
N?\\
\mbox{If such functions exist, describe all of them}.
\end{array}
\end{equation}
 The Schur
problem is often called the Carath\'eodory or the
Carath\'eodory-Fej\'{e}r problem. This problem
was posed and solved by I.~Schur in \cite{Schur} for scalar case
($\sM=\sN=\dC$) and later this and other interpolation problems for
matrix and operator cases attracted attention of many authors, which
used various methods to solve them, see \cite{Alpay, ArovKrein, BC,
Const2, DFK, Dym, FoFr, FoFrGoKa, FrKirLas, Kh3, Kh2, Kh, Peller} and
references therein. It is proved in \cite{Schur} for $\sM=\sN=\dC$
that solutions of the interpolation problem \eqref{SCHPRB} exist if
and only if the lower triangular Toeplitz matrix
\begin{equation}\label{toepn}
T_N:=\begin{bmatrix}C_0&0&0&\ldots&0\cr C_1&C_0&0&\ldots&0\cr
\vdots&\vdots&\vdots&\vdots&\vdots\cr
C_N&C_{N-1}&C_{N-2}&\ldots&C_0\end{bmatrix}
\end{equation}
is a contraction in $\dC^{N+1}$ with the standard inner product,
i.e., $I-T^*_NT_N\ge 0$. The uniqueness holds if and only if
$$\det(I-T^*_NT_N)=0.$$
In the non-uniqueness case
 a descriptions of all
solutions in \cite{Schur} is given in the form of fractional-linear
transformation
\begin{equation}
\label{Schsol}
\Theta(z)=\cfrac{e_N(z)\cE(z)+f_{N}(z)}{g_{N}(z)\cE(z)+h_{N}(z)},
\end{equation}
were $\cE$ is an arbitrary scalar Schur class function and
$$e_N(z),\; f_N(z), \;g_N(z),\; h_N(z)$$
are polynomials. The approach proposed by Schur is based on the
transformation
\[
 \bS\ni f\mapsto
\f(z):=\frac{f(z)-f(0)}{z(1-\overline{f(0)}f(z))}\in\bS,\; z\in\dD.
\]
The successive application of this transform to $f\in\bS$
\[ f_0(z)=f(z),\ldots,
f_{k+1}(z)=\frac{f_k(z)-f_k(0)}{z(1-\overline{f_k(0)}f_k(z))},\ldots
\]
is called nowadays the \textit{Schur algorithm}. If $f(z)$ is not a
finite Blaschke product, then $\{f_k\}$ is an infinite sequence of
Schur functions called the \textit{associated functions} and none of
them is a finite Blaschke product. The numbers $\gamma_k:=f_k(0)$
are called the {\it Schur parameters}. Note that
\[
f_k(z)=\frac{\gamma_k+z f_{k+1}(z)}{1+\bar\gamma_kz
f_{k+1}(z)}=\gamma_k+(1-|\gamma_k|^2)\frac{z
f_{k+1}(z)}{1+\bar\gamma_kz f_{k+1}(z)},\; k\in\dN_0.
\]
In the case when
\[f(z)=e^{i\f}\prod_{k=1}^l \frac{z-z_k}{1-\bar z_k z}
\]
is a finite Blaschke product of order $l$, the Schur algorithm
terminates at the $l$-th step, i.e., the sequence of Schur
parameters $\{\gamma_n\}_{k=0}^l$ is finite, $|\gamma_k|<1$ for
$n=0,1,\ldots,l-1$, and $|\gamma_l|=1$. Schur proved that $f$ can be
uniquely recovered from $\{\gamma_k\}$, i.e., there is-one-to one
correspondence between Schur class functions and their Schur
parameters.

 The coefficient matrix in \eqref{Schsol}
\[
W_N(z)=\begin{bmatrix} e_N(z)&
g_N(z)\cr f_N(z)&h_N(z)
\end{bmatrix}
\]
can be calculated inductively ($N+1$ steps) by means of Schur
parameters
$$\gamma_0,\gamma_1,\ldots\gamma_N,$$
obtained by the Schur algorithm which starts with
$$f_0(z)=C_0+C_1z+\cdots+C_N z^N.$$
In the case of uniqueness the solution $\Theta$ can be obtained as
follows \cite{Schur}:
\begin{itemize}
\item find $m\in\dN_0$, $m\le N$, such that
$$\det(I-T^*_{m-1}T_{m-1})\ne 0\quad\mbox{and}\quad\det(I-T^*_{m}T_{m})= 0,$$
\item calculate
\[
\Theta_{m-1}(z)=\frac{\gamma_{m-1}+z \gamma_{m}}{1+z\bar\gamma_{m-1}
\gamma_{m}},\ldots, \Theta_0=\frac{\gamma_{0}+z
\Theta_{1}(z)}{1+z\bar\gamma_0 \Theta_{1}(z)},\; z\in\dD
\]
\item the function $\Theta(z)=\Theta_0(z),\; z\in\dD$ is the solution.
\end{itemize}
For matrix and operator cases parameterizations similar to
\eqref{Schsol}
  can be found in \cite{Arov1995, BC, DFK, FrKirLas, FoFr,
FoFrGoKa}. As has been mentioned before, different methods were
used. \textit{The existence criteria in the operator case is similar
to the scalar one: the Toeplitz block operator matrix  $T_N$ given
by \eqref{toepn} is a contraction, acting form $\sM^{N+1}$ into
$\sN^{N+1}.$}

 In the present
paper a new approach to a parametrization of all solutions to the
Schur problem for operator valued functions is suggested. We
essentially use the properties of the block operator CMV
(Cantero--Moral--Vel\'azquez) matrices \cite{ArlMFAT2009} and their
sub-matrices. The scalar CMV matrices ($\sM=\sN=\dC$) appeared in
the theory of scalar orthogonal polynomials on the unit circle
$\dT=\{z\in\dC:|z|=1\}$ \cite{CMV1}, \cite{Si1}. In
\cite{ArlMFAT2009} the block operator CMV matrices were defined,
studied, and applied to the theory of conservative discrete
time-invariant linear systems and to the dilation theory. Block
operator CMV matrices are built by means of an arbitrary choice
sequence and, in particular, by means of the Schur parameters of an
arbitrary $\Theta\in\bS(\sM,\sN)$. It is established in
\cite{ArlMFAT2009} that the simple conservative systems associated
with block operator CMV matrix are realizations of given function
from $\bS(\sM,\sN)$, while the truncated block operator CMV matrices
are the models of completely non-unitary contractions
\cite{ArlMFAT2009}, \cite{AGT}. We will use finite
sub--CMV--matrices which take three-diagonal block operator form and
which are constructed by means of the finite choice sequences
associated with the data of solvable Schur problems. It should be
mentioned that in our paper \cite{MathNach_2012} the similar
approach, using three-diagonal matrices and the corresponding
resolvent formulas, has been applied to the descriptions of all
solutions to the operator truncated Hamburger moment problem.

The paper is organized as follows. In Section \ref{pre} we describe
the Schur algorithm for operator-valued Schur class functions and
the relationships between the Schur parameters, the Taylor
coefficients, the Kre\u\i n shorted operators, and lower triangular
Toeplitz matrices. In Section \ref{cmv} we recall (see
\cite{ArlMFAT2009}) the constructions of block operator CMV and
truncated CMV matrices for a choice sequence $\{\Gamma_n\}$ in the
case $\sD_{\Gamma_k}\ne\{0\}$, $\sD_{\Gamma^*_k}\ne\{0\}$ for all
$k.$ In Appendix A we present the explicit form of the block
operator CMV and truncated CMV matrices in the cases when $\Gamma_m$
is isometry, co-isometry, unitary for some $m$. The core of the
paper is Section \ref{svyasi} where we derive useful resolvent
formulas (Proposition \ref{resform} and Theorem \ref{formula1}) and,
using the three-diagonal finite sub-matrices of truncated CMV
matrices, we establish explicit connections between the function
$\Theta\in\bS(\sM,\sN)$ and the functions $\Theta_k$ associated with
$\Theta$ in accordance with the Schur algorithm (Theorems
\ref{desr}, \ref{connect21}). In Section \ref{reshen} we apply
results obtained in Section \ref{svyasi} to the parametrization of
all solutions to the Schur problem.

\section{Preliminaries}\label{pre}
\subsection{The Schur algorithm for operator-valued functions}
 The Schur algorithm for matrix and operator valued Schur class functions  has
been considered in \cite{DGK, DFK,BC, Const, Const2, FoFr}. It is
based on the following theorem which goes back to \cite{Shmul1,
Shmul2}:
\begin{theorem}
Let $\sM$ and $\sN$ be Hilbert spaces and let the function
$\Theta(z)$ be from the Schur class ${\bf S}(\sM,\sN).$ Then there
exists a function $\cZ(z)$ from the Schur class ${\bf
S}(\sD_{\Theta(0)},\sD_{\Theta^*(0)})$ such that
\begin{equation}
\label{MREP}
\Theta(z)=\Theta(0)+D_{\Theta^*(0)}\cZ(z)(I+\Theta^*(0)\cZ(z))^{-1}D_{\Theta(0)},\;z\in\dD.
\end{equation}
\end{theorem}
The representation \eqref{MREP} of a function $\Theta(z)$ from the
Schur class is called \textit{the M\"obius representation} of
$\Theta(z)$ and the function $\cZ(z)$ is called \textit{the M\"obius
parameter} of $\Theta(z)$. Clearly, $\cZ(0)=0$ and from Schwartz's
lemma one obtains that
\[
z^{-1}\cZ(z)\in\bS(\sD_{\Theta(0)},\sD_{\Theta^*(0)}).
\]
\textbf{The operator Schur's algorithm \cite{BC}}. For
$\Theta\in{\bf S}(\sM,\sN)$ put $\Theta_0(z)=\Theta(z)$ and let
$\cZ_0(z)$ be the M\"obius parameter of $\Theta$. Define
\[
\Gamma_0=\Theta(0),\; \Theta_1(z)=z^{-1}\cZ_0(z)\in {\bf
S}(\sD_{\Gamma_0},\sD_{\Gamma^*_0}),\;\Gamma_1=
\Theta_1(0)=\cZ'_0(0).
\]
If $\Theta_0(z),\ldots,\Theta_k(z)$ and $\Gamma_0,\ldots, \Gamma_k$
already defined, then let $\cZ_{k+1}\in {\bf
S}(\sD_{\Gamma_k},\sD_{\Gamma^*_k})$ be the M\"obius parameter of
$\Theta_k$. Put
\[
\Theta_{k+1}(z)=z^{-1}\cZ_{k+1}(z),\; \Gamma_{k+1}=\Theta_{k+1}(0).
\]
 The contractions $\Gamma_0\in\bL(\sM,\sN),$
$\Gamma_k\in\bL(\sD_{\Gamma_{k-1}},\sD_{\Gamma^*_{k-1}})$,
$k=1,2,\ldots$ are called the \textit{Schur parameters} of $\Theta$
and the function $\Theta_k \in {\bf
S}(\sD_{\Gamma_{k-1}},\sD_{\Gamma^*_{k-1}})$ is called the $k-th$
\textit{associated function}. Thus,
\[
\begin{array}{l}
\Theta_k(z)=\Gamma_k+z
D_{\Gamma^*_k}\Theta_{k+1}(z)(I+z\Gamma^*_k\Theta_{k+1}(z))^{-1}D_{\Gamma_k}\\
=\Gamma_k+z
D_{\Gamma^*_k}(I+z\Theta_{k+1}(z)\Gamma^*_k)^{-1}\Theta_{k+1}(z)D_{\Gamma_k},
\;z\in\dD,
\end{array}
\]
and
\[
\Theta_{k+1}(z)\uphar\ran
D_{\Gamma_k}=z^{-1}D_{\Gamma^*_k}(I-\Theta_k(z)\Gamma^*_k)^{-1}
(\Theta_k(z)-\Gamma_k)D^{-1}_{\Gamma_k}\uphar\ran D_{\Gamma_k}.
\]
Clearly, the sequence of Schur parameters $\{\Gamma_k\}$ is infinite
if and only if the operators $\Gamma_k$ are non-unitary. The
sequence of Schur parameters consists of finite number of operators
$\Gamma_0,$ $\Gamma_1,\ldots, \Gamma_N$ if and only if
$\Gamma_N\in\bL(\sD_{\Gamma_{N-1}},\sD_{\Gamma^*_{N-1}})$ is
unitary. If $\Gamma_N$ is non-unitary but isometric (respect.,
co-isometric), then $\Gamma_k=0\in\bL(0,\sD_{\Gamma^*_N})$
(respect., $\Gamma_k=0\in\bL(\sD_{\Gamma_N},0)$) for all $k>N$. The
following theorem \cite{BC, Const} is the operator generalization of
Schur's result.
\begin{theorem} \label{SchurAlg} There is a one-to-one
correspondence between the Schur class ${\bf S}(\sM,\sN)$ and the
set of all sequences of contractions $\{\Gamma_k\}_{k\ge 0}$ such
that
\begin{equation}
\label{CHSEQ} \Gamma_0\in\bL(\sM,\sN),\;\Gamma_k\in
\bL(\sD_{\Gamma_{k-1}},\sD_{\Gamma^*_{k-1}}),\; k\ge 1.
\end{equation}
\end{theorem}
Notice that a sequence of contractions of the form \eqref{CHSEQ} is
called the \textit{choice sequence} \cite{CF}.
 There are connections, established in \cite{Const}, between the
Taylor coefficients $\{C_n\}_{n\ge 0}$ and Schur parameters of
$\Theta\in\bS(\sM,\sN)$. These connections are given by the
relations
\begin{equation}
\label{ConstForm}
\begin{array}{l}
C_0=\Gamma_0,\\
 C_n={formula}_n(\Gamma_0,\Gamma_1,\cdots,
\Gamma_{n-1})+\\
\qquad \qquad D_{\Gamma^*_0}D_{\Gamma^*_1}\cdots D_{\Gamma^*_{n-1}}
\Gamma_n D_{\Gamma_{n-1}}\cdots D_{\Gamma_1}D_{\Gamma_0},\; n\ge 1.
\end{array}
\end{equation}
 Here ${formula}_n(\Gamma_0,\Gamma_1,\cdots, \Gamma_{n-1})$
is some expression, depending on $\Gamma_0,\Gamma_1,\cdots,
\Gamma_{n-1}$. Let now $\{C_k\}_{k=0}^\infty$ be a sequence of
operators from $\bL(\sM,\sN)$. Then (\cite[Theorem 2.1]{BC}) there
is a one-to-one correspondence between the set of contractions
\[
 T_\infty:=\begin{bmatrix}C_0&0&0&0&0&\ldots\cr
C_1&C_0&0&0&0&\ldots\cr C_2&C_1&C_0&0&0&\ldots\cr
C_3&C_2&C_1&C_0&0&\ldots\cr
\vdots&\vdots&\vdots&\vdots&\vdots&\vdots
\end{bmatrix}: \ell_2(\sM)\to \ell_2(\sN)
\]
and the set of all choice sequences $\Gamma_0\in\bL(\sM,\sN),$
$\Gamma_k\in\bL(\sD_{\Gamma_{k-1}},\sD_{\Gamma^*_{k-1}})$,
$k=1,\ldots$. The connections between $\{C_k\}$ and $\{\Gamma_k\}$
are also given by \eqref{ConstForm}. The operators $\{\Gamma_k\}$
can be successively defined \cite[proof of Theorem 2.1]{BC},
 using parametrization of contractive
block-operator matrices, from the matrices
\[
T_0=C_0=\Gamma_0, \;T_1=\begin{bmatrix}C_0&0\cr
C_1&C_0\end{bmatrix}, \;T_2=\begin{bmatrix}C_0&0&0\cr C_1&C_0&0\cr
C_2&C_1&C_0\end{bmatrix},\ldots.
\]
 Moreover, $T_\infty=T_\Theta$,  $\Theta(\lambda)=\sum\limits_{n=0}^{\infty}\lambda^n C_n$,
$\lambda\in\dD$, and $\{\Gamma_k\}_{k\ge 0}$ are the Schur
parameters of $\Theta$ \cite[Proposition 2.2]{BC}. Put
\[
\wt\Theta(\lambda):=\Theta^*(\bar \lambda),\;|\lambda|<1.
\]
Then $ \wt\Theta(\lambda)=\sum\limits_{n=0}^{\infty}\lambda^n
C^*_n.$ Clearly, if $\{\Gamma_0,\Gamma_1,\ldots\}$ are the Schur
parameters of $\Theta$, then $\{\Gamma^*_0,\Gamma^*_1,\ldots\}$ are
the Schur parameters of $\wt \Theta$.

Besides $T_N$ we will consider the following lower triangular block
operator matrices from $\bL(\sN^{N+1},\sM^{N+1})$:
\begin{equation}
\label{wttoepn} \wt T_N=\wt T_N(C_0,C_1,\ldots,
C_N):=\begin{bmatrix}C^*_0&0&0&\ldots&0\cr C^*_1&C^*_0&0&\ldots&0\cr
\vdots&\vdots&\vdots&\vdots&\vdots\cr
C^*_N&C^*_{N-1}&C^*_{N-2}&\ldots&C^*_0\end{bmatrix}.
\end{equation}
If $T_N$ of the form \eqref{toepn} is a contraction, then operators
$\{C_k\}_{k=0}^N$ are said to be \textit{the Schur sequence}
\cite{DFK}.

\subsection{The Kre\u{\i}n shorted operator and Toeplitz matrices}

For every nonnegative bounded operator $S$ in the Hilbert space
$\cH$ and every subspace $\cK\subset \cH$ M.G.~Kre\u{\i}n \cite{Kr}
defined the operator $S_{\cK}$ by the relation
\[
 S_{\cK}=\max\left\{\,Z\in \bL(\cH):\,
    0\le Z\le S, \, {\ran}Z\subseteq{\cK}\,\right\}.
\]
The equivalent definition is:
\[
 \left(S_{\cK}f, f\right)=\inf\limits_{\f\in \cK^\perp}\left\{\left(S(f + \varphi),f +
 \varphi\right)\right\},
\quad  f\in\cH.
\]
Here $\cK^\perp:=\cH\ominus{\cK}$. The properties of $S_{\cK}$, were
studied by M.~Kre\u{\i}n and by other authors (see \cite{ARL1} and
references therein).
 $S_{\cK}$ is called the \textit{shorted
operator} (see \cite{And, AT}). Let the subspace $\Omega$ be defined
as follows
\[
 \Omega=\{\,f\in \cran S:\,S^{1/2}f\in {\cK}\,\}=\cran S\ominus
S^{1/2}\cK^\perp.
\]
It is proved in \cite{Kr} that $S_{\cK}$ takes the form
\[
 S_{\cK}=S^{1/2}P_{\Omega}S^{1/2}.
\]
 Hence, $\ker S_\cK\supseteq\cK^\perp.$
Moreover \cite{Kr},
\[
 {\ran}S_{\cK}^{1/2}={\ran}S^{1/2}\cap{\cK}.
\]
It follows that \[
  S_{\cK}=0 \iff  \ran S^{1/2}\cap \cK=\{0\}.
\]
We identify $\sM$ ($\sN$, respectively) with the subspace
\[
\sM\oplus\underbrace{\{0\}\oplus\{0\}\oplus\cdots\oplus\{0\}}_n\quad\left(\sN\oplus\underbrace{\{0\}\oplus\{0\}\oplus\cdots\oplus\{0\}}_n\right)
\]
in $\sM^{n+1}$ ($\sN^{n+1}$), and with
$\sM\oplus\bigoplus\limits_{k=1}^\infty\{0\}\quad\left(\sN\oplus\bigoplus\limits_{k=1}^\infty\{0\}\right)$
in $l_2(\sM)$ ($l_2(\sN)).$ The next relations are established in
\cite{ArlIEOT2011}.
\begin{theorem}
\label{Main44} Let $\Theta\in\bS(\sM,\sN)$ and let
$\{\Gamma_0,\Gamma_1,\cdots\}$ be the Schur parameters of $\Theta$.
Then for each $n$ the relations
\[
\left(D^2_{T_{\Theta,n}}\right)_{\sM}=D_{\Gamma_0}D_{\Gamma_1}\cdots
D_{\Gamma_{n-1}} D^2_{\Gamma_n} D_{\Gamma_{n-1}}\cdots
D_{\Gamma_1}D_{\Gamma_0}P_\sM
\]
\[
\left(D^2_{T_{\wt\Theta,n}}\right)_{\sN}=D_{\Gamma^*_0}D_{\Gamma^*_1}\cdots
D_{\Gamma^*_{n-1}} D^2_{\Gamma^*_n} D_{\Gamma^*_{n-1}}\cdots
D_{\Gamma^*_1}D_{\Gamma^*_0}P_\sN,\\
\]
hold. Moreover
\[
\left(D^2_{T_{\Theta}}\right)_{\sM}=s-\lim\limits_{n\to\infty}\left(D_{\Gamma_0}D_{\Gamma_1}\cdots
D_{\Gamma_{n-1}} D^2_{\Gamma_n} D_{\Gamma_{n-1}}\cdots
D_{\Gamma_1}D_{\Gamma_0}\right)P_\sM,
\]
\[
\left(D^2_{T_{\wt\Theta}}\right)_{\sN}=s-\lim\limits_{n\to\infty}\left(D_{\Gamma^*_0}D_{\Gamma^*_1}\cdots
D_{\Gamma^*_{n-1}} D^2_{\Gamma^*_n} D_{\Gamma^*_{n-1}}\cdots
D_{\Gamma^*_1}D_{\Gamma^*_0}\right)P_\sN.
\]
\end{theorem}

By means of relations \eqref{ConstForm} contractions $T_0,
T_1,\ldots, T_N$ determine choice parameters
\[
\Gamma_0:=C_0,\;
\Gamma_1\in\bL(\sD_{\Gamma_0},\sD_{\Gamma_0^*}),\ldots,\Gamma_N\in\bL(\sD_{\Gamma_{N-1}},\sD_{\Gamma_{N-1}^*})
\]
and for operators $T_n$ and $\wt T_n$ ($n\le N$) the relations
\cite{ArlIEOT2011}:
\begin{equation}
\label{krshoren}
\begin{array}{l}
\left(D^2_{T_{n}}\right)_{\sM}=D_{\Gamma_0}D_{\Gamma_1}\cdots
D_{\Gamma_{n-1}} D^2_{\Gamma_n} D_{\Gamma_{n-1}}\cdots
D_{\Gamma_1}D_{\Gamma_0}P_\sM,\\
\left(D^2_{\wt
T_{n}}\right)_{\sN}=D_{\Gamma^*_0}D_{\Gamma^*_1}\cdots
D_{\Gamma^*_{n-1}} D^2_{\Gamma^*_n} D_{\Gamma^*_{n-1}}\cdots
D_{\Gamma^*_1}D_{\Gamma^*_0}P_\sN
\end{array}
\end{equation}
hold true. The next result can be found in \cite[Theorem 2.6]{BC}.
\begin{theorem}
\label{SchPr2}  Consider a solvable Schur problem with the data
$$C_0,\ldots, C_N \in\bL(\sM,\sN).$$
Then the solution is unique if and only if the corresponding choice
parameters $\{\Gamma_n\}_{n=0}^N$, determined by the operator $T_N$,
satisfy the condition: one of $\Gamma_n$, $0\le n\le N$ is an
isometry or a co-isometry.
\end{theorem}

 The next statement is established in \cite{ArlIEOT2011}.
\begin{theorem}
\label{OPUNI} Let the data $C_0,C_1,\ldots, C_N\in\bL(\sM,\sN)$ be
the Schur sequence. Then the following statements are equivalent
\begin{enumerate}
\def\labelenumi{\rm (\roman{enumi})}
\item the Schur problem has a unique solution;
\item either $(D^2_{T_N})_\sM=0$ or $(D^2_{\wt T_N})_\sN=0$;
\item either $\sM\cap \ran D_{T_N}=\{0\}$ or $\sN\cap \ran D_{\wt T_N}=\{0\}$.
\end{enumerate}
\end{theorem}
\subsection{The Schur sequences and the Schur
parameters}\label{shseq} Let the Schur sequence
$C_0,\ldots,C_N\subset\bL(\sM,\sN)$, satisfying the conditions
\begin{equation}
\label{nonun1} (D^2_{T_N})_\sM\ne 0,\; (D^2_{\wt T_N})_\sN\ne 0,
\end{equation}
be given. Then the corresponding Schur parameters
\[
\Gamma_0\in\bL(\sM,\sN),\;\ldots,\Gamma_N\in\bL(\sD_{\Gamma_{N-1}},\sD_{\Gamma^*_{N-1}}),
\]
satisfy the conditions
\begin{equation} \label{nonun2}
\sD_{\Gamma_N}\ne\{0\},\; \sD_{\Gamma^*_N}\ne\{0\}.
\end{equation}
One more way to find $\{\Gamma_k\}_{k=0}^N$ goes back to Schur
\cite{Schur} (see also \cite[page 448]{FoFr}). We shall describe it.
Let $\Theta\in\bS(\sM,\sN)$ be a solution to the Schur problem with
data $C_0,\ldots,C_N$. By definition $\Gamma_0=C_0$. Due to
\eqref{krshoren} and \eqref{nonun1} we have $\sD_{\Gamma_0}\ne\{0\}$
and $\sD_{\Gamma^*_0}\ne\{0\}$. If
$\Theta_1\in\bS(\sD_{\Gamma_0},\sD_{\Gamma^*_0})$  is the first
function associated with $\Theta$ and if
\begin{equation}
\label{F} F_1(z):=z(I+z\Theta_1(z)\Gamma_0^*)^{-1}\Theta_1(z),\;
z\in\dD,
\end{equation}
then the Taylor expansion
\[
F_1(z)=\sum\limits_{k=1}^\infty z^kB^{(1)}_k,\;
B_k\in\bL(\sD_{\Gamma_0},\sD_{\Gamma^*_0}),\; k\ge 1,
\]
and the relation
$\Theta(z)-\Gamma_0=D_{\Gamma^*_0}F_1(z)D_{\Gamma_0}$ lead to the
equalities
\[
C_k=D_{\Gamma^*_0}B^{(1)}_kD_{\Gamma_0},\; k\ge 1.
\]
Let
\[
\Theta_1(z)=\sum\limits_{k=0}^\infty z^k C^{(1)}_k,\;
C^{(1)}_k\in\bL(\sD_{\Gamma_0},\sD_{\Gamma^*_0}),\; k\ge 0
\]
be the Taylor expansion of $\Theta_1$.  Then from \eqref{F} we get
the equalities
\begin{equation}
\label{systequat1} \left\{\begin{array}{l}
C^{(1)}_0=B^{(1)}_1,\\
C^{(1)}_1=B^{(1)}_2+C^{(1)}_0\Gamma^*_0B^{(1)}_1,\\
\ldots\ldots\ldots\ldots\ldots\ldots\ldots\ldots\\
 C^{(1)}_{N-1}=B^{(1)}_{N}+C^{(1)}_0\Gamma^*_0B^{(1)}_{N-1}+C^{(1)}_1\Gamma^*_0B^{(1)}_{N-2}+\cdots
+C^{(1)}_{N-2}\Gamma^*_0B^{(1)}_1.\\
\end{array}
\right.
\end{equation}
The system \eqref{systequat1} can be rewritten as follows
\begin{equation}
\label{systequat2} \left\{ \begin{array}{l}
 B^{(1)}_1=C^{(1)}_0,\\
B^{(1)}_2=C^{(1)}_1-C^{(1)}_0\Gamma^*_0B^{(1)}_1,\\
\ldots\ldots\ldots\ldots\ldots\ldots\ldots\ldots,\\
B^{(1)}_{N}=C^{(1)}_{N-1}-C^{(1)}_0\Gamma^*_0B^{(1)}_{N-1}-C^{(1)}_1\Gamma^*_0B^{(1)}_{N-2}-\cdots
-C^{(1)}_{N-2}\Gamma^*_0B^{(1)}_1.\\
\end{array}
\right.
\end{equation}
It follows that the data $C_0, C_1,\ldots, C_N$ produce $\Gamma_0$,
$B^{(1)}_1,\ldots,B^{(1)}_N$ and then
$$\Gamma_1=C^{(1)}_0,C^{(1)}_1,\ldots, C^{(1)}_{N-1}.$$
Arguing similarly for $\Theta_2,\ldots,\Theta_N$ we get
$\Gamma_2=C^{(2)}_0=\Theta_2(0),\ldots$,
$\Gamma_N=C^{(N)}_0=\Theta_N(0)$ and
\[
\sD_{\Gamma_j}\ne\{0\},\;\sD_{\Gamma^*_j}\ne\{0\},\; j=1,\ldots, N
\]
Conversely, given the choice sequence $\Gamma_0,\ldots,$ $\Gamma_N$
satisfying \eqref{nonun2}. Then
\[
C^{(N)}_0=\Gamma_N,\;B^{(N)}_1=C^{(N)}_0,\;
C^{(N-1)}_1=D_{\Gamma^*_{N-1}}B^{(N)}_1D_{\Gamma_{N-1}},\;
C^{(N-1)}_0=\Gamma_{N-1}.
\]
Using equation of the type \eqref{systequat2} we find  $B^{(N-1)}_1$
and $B^{(N-1)}_2$ and then find
\[
C^{(N-2)}_1=D_{\Gamma^*_{N-2}}
B^{(N-1)}_1D_{\Gamma_{N-2}},\;C^{(N-2)}_2=D_{\Gamma^*_{N-2}}
B^{(N-1)}_2D_{\Gamma_{N-2}.}
\]
Finally we find the Schur sequence $C^{(0)}_0, \ldots C^{(0)}_N$.
Clearly it satisfies \eqref{nonun1}.

Notice that
 in \cite[Theorem 2.2]{ArlMFAT2009} it is proved that if
$\Theta,\,\wh\Theta\in\bS(\sM,\sN)$, $\{\Gamma_k\}$,
$\{\wh\Gamma_k\}$ are the Schur parameters of $\Theta$ and
$\wh\Theta$, correspondingly, and
\[
\Gamma_0=\wh\Gamma_0,\ldots,\Gamma_N=\wh\Gamma_N,
\]
then
\[
||\Theta(z)-\wh\Theta(z)||=o(|z|^{N+1}),\; z\to 0.
\]

Suppose now that the Schur sequence
$\{C_k\}_{k=0}^N\subset\bL(\sM,\sN)$ is such that
$$\left
(D^2_{T_N}\right)_\sM=0.
$$
Then by Theorem \ref{OPUNI} the Schur problem has a unique solution
$\Theta$. From \eqref{krshoren} it follows that there is a number
$p$, $p\le N$ such that $\left(D^2_{T_{p-1}}\right)_\sM\ne 0$ but
$\left(D^2_{T_p}\right)_\sM=0$, i.e., the Schur parameters of
$\Theta$ are
\[
\begin{array}{l}
\Gamma_0,\ldots, \Gamma_{p-1},\Gamma_p,\;
\sD_{\Gamma_p}=0,\;\sD_{\Gamma^*_p}\ne\{0\},\\
\Gamma_{p+l}=0\in\bL(0,\sD_{\Gamma^*_p}),\; l\ge 1.
\end{array}
\]
The operators $\Gamma_0,\ldots,\Gamma_p$ can be found by the
procedure described above.

In the case $\left(D^2_{\wt T_N}\right)_\sN= 0$ we proceed
similarly.

\section{Block operator CMV matrices}\label{cmv} Let
\[
\Gamma_0\in\bL(\sM,\sN),\;\Gamma_k\in
\bL(\sD_{\Gamma_{k-1}},\sD_{\Gamma^*_{k-1}}),\; k\in\dN
\]
be a choice sequence. There are two unitary and unitarily equivalent
block operator CMV matrices corresponding to the choice sequence
$\{\Gamma_n\}$ \cite{ArlMFAT2009}. We briefly describe their
constructions and their forms for the case
$\sD_{\Gamma_k}\ne\{0\}$, $\sD_{\Gamma^*_k}\ne\{0\}$  for all  $k$,  
i.e., $\sD_{\Gamma_k}\ne\{0\}$, $\sD_{\Gamma^*_k}\ne\{0\}$ for each
$k$. The rest cases we consider in detail in Appendix A.
\subsection{CMV matrices}

Recall that if $\{\cH_k\}_{k=1}^\infty$ be a given sequence of
Hilbert spaces, then
\[
H_N=\bigoplus\limits_{k=1}^N  \cH_k
\]
is the Hilbert space with the inner product
$(f,g)=\sum\limits_{k=1}^N(f_k,g_k)_{\cH_k}$ for $f=(f_1, \ldots,
f_N)^T$ and $g=(g_1, \ldots, g_N)^T$, $f_k, g_k\in \cH_k,$
$k=1,\ldots,N$ and the norm $
||f||^2=\sum\limits_{k=1}^N||f_k||^2_{\cH_k}.$
 The Hilbert space
\[
H_{\infty}= \bigoplus\limits_{k=1}^\infty \cH_k
\]
consists of all vectors of the form $f=(f_1,f_2, \ldots)^T,$ $f_k\in
\cH_k$, $k=1,2,\ldots,$ such that
\[
||f||^2=\sum\limits_{k=1}^\infty ||f_k||^2_{\cH_k}<\infty.
\]
The inner product is given by $
(f,g)=\sum\limits_{k=1}^\infty(f_k,g_k)_{\cH_k}.$

Define the Hilbert spaces
\[
\sH_0=\sH_0(\{\Gamma_k\}_{k\ge 0}):=\bigoplus\limits_{k\ge
0}\begin{array}{l}\sD_{\Gamma_{2k}}\\\oplus\\\sD_{\Gamma^*_{2k+1}}\end{array},\;
\wt\sH_0=\wt\sH_0(\{\Gamma_k\}_{k\ge 0}):=\bigoplus\limits_{k\ge
0}\begin{array}{l}\sD_{\Gamma^*_{2k}}\\\oplus\\\sD_{\Gamma_{2k+1}}\end{array}.
\]
From these definitions it follows, that
\[
\wt\sH_0(\{\Gamma^*_k\}_{k\ge 0})=\sH_0(\{\Gamma_k\}_{k\ge 0}),\;
\sH_0(\{\Gamma^*_k\}_{k\ge 0})=\wt\sH_0(\{\Gamma_k\}_{k\ge 0}).
\]
The spaces $\sN\bigoplus\sH_0$ and $\sM\bigoplus\wt\sH_0$ we
represent in the form
\[
\sN\bigoplus\sH_0=\begin{array}{l}\sN\\\oplus\\\sD_{\Gamma_0}\end{array}\bigoplus\limits_{k\ge
1}\begin{array}{l}\sD_{\Gamma^*_{2k-1}}\\\oplus\\\sD_{\Gamma_{2k}}\end{array},\;
\sM\bigoplus\wt\sH_0=\begin{array}{l}\sM\\\oplus\\\sD_{\Gamma^*_0}\end{array}
\bigoplus\limits_{k\ge
1}\begin{array}{l}\sD_{\Gamma_{2k-1}}\\\oplus\\\sD_{\Gamma^*_{2k}}\end{array}.
\]
Let
\[\begin{array}{l}
{\bf J}_{\Gamma_0}=\begin{bmatrix} \Gamma_0& D_{\Gamma^*_0}\cr
D_{\Gamma_0}&-\Gamma^*_0\end{bmatrix}:\begin{array}{l}\sM\\\oplus\\\sD_{\Gamma^*_0}\end{array}\to
\begin{array}{l}\sN\\\oplus\\\sD_{\Gamma_{0}}\end{array},\;
{\bf J}_{\Gamma_k}=\begin{bmatrix} \Gamma_k& D_{\Gamma^*_k}\cr
D_{\Gamma_k}&-\Gamma^*_k\end{bmatrix}:\begin{array}{l}\sD_{\Gamma_{k-1}}\\\oplus\\\sD_{\Gamma^*_k}\end{array}\to
\begin{array}{l}\sD_{\Gamma^*_{k-1}}\\\oplus\\\sD_{\Gamma_{k}}\end{array},\;
k\in\dN.
\end{array}
\]
 be the elementary rotations. Define the following unitary operators
\begin{equation}\label{OLM}
\begin{array}{l}
\cM_0=\cM_0(\{\Gamma_k\}_{k\ge 0}):=I_\sM\bigoplus\limits_{k\ge
1}{\bf
J}_{\Gamma_{2k-1}}:\sM\bigoplus\sH_0\to \sM\bigoplus\wt\sH_0,\\
\wt \cM_0=\wt\cM_0(\{\Gamma_k\}_{k\ge
0}):=I_\sN\bigoplus\limits_{k\ge 1}{\bf
J}_{\Gamma_{2n-1}}:\sN\bigoplus\sH_0\to\sN\bigoplus\wt\sH_0,\\
\cL_0=\cL_0(\{\Gamma_k\}_{k\ge 0}):={\bf
J}_{\Gamma_0}\bigoplus\limits_{k\ge 1}{\bf
J}_{\Gamma_{2n}}:\sM\bigoplus\wt\sH_0\to\sN\bigoplus\sH_0.\\
\end{array}
\end{equation}
Observe that
$\left(\cL_0(\{\Gamma_k\}_{k\ge
0})\right)^*=\cL_0(\{\Gamma^*_k\}_{k\ge 0}).$
 Let
\begin{equation}
\label{V0} \cV_0=\cV_0(\{\Gamma_k\}_{k\ge
0}):=\bigoplus\limits_{k\ge 1}{\bf
J}_{\Gamma_{2k-1}}:\sH_0\to\wt\sH_0.
\end{equation}
Clearly, the operator $\cV_0$ is unitary and
\begin{equation}
\label{MOV}
\cM_0=I_\sM\bigoplus\cV_0,\;\wt\cM_0=I_\sN\bigoplus\cV_0.
\end{equation}
It follows that
$\left(\wt\cM_0(\{\Gamma_k\}_{k\ge
0})\right)^*=\cM_0(\{\Gamma^*_k\}_{k\ge 0}),$
$\left(\cM_0(\{\Gamma_k\}_{k\ge
0})\right)^*=\wt\cM_0(\{\Gamma^*_k\}_{k\ge 0}).$
Finally, define the unitary operators
\begin{equation}
\label{defcmv}
\begin{array}{l}
\cU_0=\cU_0(\{\Gamma_k\}_{k\ge
0}):=\cL_0\cM_0:\sM\bigoplus\sH_0\to\sN\bigoplus\sH_0,\\
\wt\cU_0=\wt\cU_0(\{\Gamma_k\}_{k\ge
0}):=\wt\cM_0\cL_0:\sM\bigoplus\wt\sH_0\to\sN\bigoplus\wt\sH_0.
\end{array}
\end{equation}
By calculations we get {\footnotesize
\begin{equation}
\label{U0} \cU_0=
\begin{bmatrix}
\Gamma_0&D_{\Gamma^*_0}\Gamma_1&D_{\Gamma^*_0}D_{\Gamma^*_1}&0&0&0&0&0&\ldots\cr
D_{\Gamma_0}&-\Gamma^*_0\Gamma_1&-\Gamma^*_0D_{\Gamma^*_1}&0&0&0&0&0&\ldots\cr
0&\Gamma_2D_{\Gamma_1}&-\Gamma_2\Gamma^*_1&D_{\Gamma^*_2}\Gamma_3&D_{\Gamma^*_2}D_{\Gamma^*_3}&0&0&0&\ldots\cr
0&D_{\Gamma_2}D_{\Gamma_1}&-D_{\Gamma_2}\Gamma^*_1&-\Gamma^*_2\Gamma_3&-\Gamma^*_2D_{\Gamma^*_3}&0&0&0&\ldots\cr
0&0&0&\Gamma_4D_{\Gamma_3}&-\Gamma_4\Gamma^*_3&D_{\Gamma^*_4}\Gamma_5&D_{\Gamma^*_4}D_{\Gamma^*_5}&0&\dots\cr
\vdots&\vdots&\vdots&\vdots&\vdots&\vdots&\vdots&\vdots&\vdots
\end{bmatrix},
\end{equation}
\begin{equation}
\label{WTU0}
 \wt\cU_0=
\begin{bmatrix}
\Gamma_0&D_{\Gamma^*_0}&0&0&0&0&0&\ldots\cr
\Gamma_1D_{\Gamma_0}&-\Gamma_1\Gamma^*_0&D_{\Gamma^*_1}\Gamma_2&D_{\Gamma^*_1}D_{\Gamma^*_2}&0&0&0&\ldots\cr
D_{\Gamma_1}D_{\Gamma_0}&-D_{\Gamma_1}\Gamma^*_0&-\Gamma^*_1\Gamma_2&-\Gamma^*_1D_{\Gamma^*_2}&0&0&0&\ldots\cr
0&0&\Gamma_3D_{\Gamma_2}&-\Gamma_3\Gamma^*_2&D_{\Gamma^*_3}\Gamma_4&D_{\Gamma^*_3}D_{\Gamma^*_4}&0&\ldots\cr
0&0&D_{\Gamma_3}D_{\Gamma_2}&-D_{\Gamma_3}\Gamma^*_2&-\Gamma^*_3\Gamma_4&-\Gamma^*_3D_{\Gamma^*_4}&0&\ldots\cr
\vdots&\vdots&\vdots&\vdots&\vdots&\vdots&\vdots&\vdots&
\end{bmatrix}.
\end{equation}
}
 The block operator matrices $\cU_0$ and $\wt\cU_0$ are called
\cite{ArlMFAT2009} \textit{block operator CMV matrices}.
 Observe that
\begin{equation}
\label{MU} \wt M_0\cU_0=\wt\cU_0\cM_0,
\end{equation}
and
\[
\left(\cU_0(\{\Gamma_k\}_{k\ge
0})\right)^*=\wt\cU_0(\{\Gamma^*_k\}_{k\ge
0}),\;\left(\wt\cU_0(\{\Gamma_k\}_{k\ge
0})\right)^*=\cU_0(\{\Gamma^*_k\}_{k\ge 0})
\]
Therefore the matrix $\wt\cU_0$ can be obtained from $\cU_0$ by
passing to the adjoint $\cU^*_0$ and then by replacing $\Gamma_k$
(respect., $\Gamma^*_k$) by $\Gamma^*_k$ (respect., $\Gamma_k$) for
all $n$. In the case when the choice sequence consists of complex
numbers from the unit disk the matrix $\wt\cU_0$ is the transpose to
$\cU_0$, i.e., $\wt\cU_0=\cU^t_0$.

The matrices $\cU_0$ and $\wt\cU_0$ admit the three-diagonal block
operator form {\footnotesize
\[
  \cU_0=
  \begin{bmatrix} \Gamma_0 & \cC_0 & 0 &0   & 0 &
\cdot &
\cdot  \\
\cA_0 & \cB_1 & \cC_1 & 0 &0& \cdot &
\cdot  \\
0    & \cA_1 & \cB_2 & \cC_2 &0& \cdot &
\cdot   \\
\vdots & \vdots & \vdots & \vdots & \vdots & \vdots & \vdots
\end{bmatrix},\; \wt\cU_0=
\begin{bmatrix} \Gamma_0 & \wt\cC_0 & 0 &0   & 0 &
\cdot &
\cdot  \\
\wt\cA_0 & \wt\cB_1 & \wt\cC_1 & 0 &0& \cdot &
\cdot  \\
0    & \wt\cA_1 & \wt\cB_2 & \wt\cC_2 &0& \cdot &
\cdot   \\
\vdots & \vdots & \vdots & \vdots & \vdots & \vdots & \vdots
\end{bmatrix},
\]
}
where
\[
\cC_0=\begin{bmatrix}D_{\Gamma^*_0}\Gamma_1&D_{\Gamma^*_0}D_{\Gamma^*_1}\end{bmatrix}:
\begin{array}{l}\sD_{\Gamma_{0}}\\\oplus\\\sD_{\Gamma^*_{1}}\end{array}
\to \sN,\;\cA_0=\begin{bmatrix} D_{\Gamma_0}\cr 0
\end{bmatrix}:\sM\to
\begin{array}{l}\sD_{\Gamma_{0}}\\\oplus\\\sD_{\Gamma^*_{1}}\end{array},
\]
\begin{equation}
\label{BLOKIT} \left\{
\begin{array}{l}
\cB_k=\begin{bmatrix}-\Gamma^*_{2k-2}\Gamma_{2k-1}&-\Gamma^*_{2k-2}D_{\Gamma^*_{2k-1}}\cr
\Gamma_{2k}D_{\Gamma_{2k-1}}&-\Gamma_{2k}\Gamma^*_{2k-1}\end{bmatrix}:\begin{array}{l}\sD_{\Gamma_{2k-2}}\\\oplus\\
\sD_{\Gamma^*_{2k-1}}\end{array}
\to\begin{array}{l}\sD_{\Gamma_{2k-2}}\\\oplus\\\sD_{\Gamma^*_{2k-1}}\end{array},\\
\cC_k=\begin{bmatrix}0&0\cr
D_{\Gamma^*_{2k}}\Gamma_{2k+1}&D_{\Gamma^*_{2k}}D_{\Gamma^*_{2k+1}}\end{bmatrix}:\begin{array}{l}\sD_{\Gamma_{2k}}
\\\oplus\\\sD_{\Gamma^*_{2k+1}}\end{array}
\to\begin{array}{l}\sD_{\Gamma_{2k-2}}\\\oplus\\\sD_{\Gamma^*_{2k-1}}\end{array},\\
\cA_k=\begin{bmatrix}D_{\Gamma_{2k}}D_{\Gamma_{2k-1}}&
-D_{\Gamma_{2k}}\Gamma^*_{2k-1}\cr
0&0\end{bmatrix}:\begin{array}{l}\sD_{\Gamma_{2k-2}}\\\oplus\\\sD_{\Gamma^*_{2k-1}}\end{array}
\to\begin{array}{l}\sD_{\Gamma_{2k}}\\\oplus\\\sD_{\Gamma^*_{2k+1}}\end{array},
\end{array}\right.
\end{equation}
\[
\wt
C_0=\begin{bmatrix}D_{\Gamma^*_0}&0\end{bmatrix}:\begin{array}{l}\sD_{\Gamma^*_0}\\\oplus\\\sD_{\Gamma_1}\end{array}\to\sN
,\;\wt\cA_0=\begin{bmatrix}\Gamma_1 D_{\Gamma_0}\cr
D_{\Gamma_1}D_{\Gamma_0}\end{bmatrix}:\sM\to\begin{array}{l}\sD_{\Gamma^*_0}\\\oplus\\\sD_{\Gamma_1}\end{array},
\]
\begin{equation}
\label{BLOKIWT} \left\{
\begin{array}{l}
 \wt\cB_k=\begin{bmatrix}
-\Gamma_{2k-1}\Gamma^*_{2k-2}&D_{\Gamma^*_{2k-1}}\Gamma_{2k}\cr
-D_{\Gamma_{2k-1}}\Gamma^*_{2k-2}&-\Gamma^*_{2k-1}\Gamma_{2k}\end{bmatrix}:\begin{array}{l}\sD_{\Gamma^*_{2k-2}}
\\\oplus\\\sD_{\Gamma_{2k-1}}\end{array}\to
\begin{array}{l}\sD_{\Gamma^*_{2k-2}}\\\oplus\\\sD_{\Gamma_{2k-1}}\end{array},\\
\wt\cC_k=\begin{bmatrix} D_{\Gamma^*_{2k-1}}D_{\Gamma^*_{2k}}&0\cr
-\Gamma^*_{2k-1}D_{\Gamma^*_{2k}}&0\end{bmatrix}:\begin{array}{l}\sD_{\Gamma^*_{2k}}\\\oplus\\\sD_{\Gamma_{2k+1}}
\end{array}\to
\begin{array}{l}\sD_{\Gamma^*_{2k-2}}\\\oplus\\\sD_{\Gamma_{2k-1}}\end{array},\\
\wt\cA_k=\begin{bmatrix}0&\Gamma_{2k+1}D_{\Gamma_{2k}}\cr
0&D_{\Gamma_{2k+1}}D_{\Gamma_{2k}}\end{bmatrix}:\begin{array}{l}\sD_{\Gamma^*_{2k-2}}\\\oplus\\\sD_{\Gamma_{2k-1}}\end{array}\to
\begin{array}{l}\sD_{\Gamma^*_{2k}}\\\oplus\\\sD_{\Gamma_{2k+1}}\end{array}.
\end{array}\right..
\end{equation}
\begin{remark}
The three-diagonal block form of the CMV matrices with scalar
entries has been established in \cite{BD}.
\end{remark}
\subsection{Truncated block operator CMV matrices}\label{sectrunc }
 Define two contractions
\begin{equation}
\label{TRUNC}\cT_0= \cT_0(\{\Gamma_k\}_{k\ge
0}):=P_{\sH_0}\cU_0\uphar\sH_0:\sH_0\to\sH_0,
\end{equation}
\begin{equation}
\label{TRUNCT} \wt\cT_0=\wt\cT_0(\{\Gamma_k\}_{k\ge
0}):=P_{\wt\sH_0}\wt\cU_0\uphar\wt\sH_0:\wt\sH_0\to\wt\sH_0.
\end{equation}
The operators $\cT_0$ and $\wt\cT_0$ take on the three-diagonal
block operator matrix forms {\footnotesize
\[
  \cT_0=\begin{bmatrix}
\cB_1 & \cC_1 & 0 &0& 0&
\cdot  \\
 \cA_1 & \cB_2 & \cC_2 &0& 0 &
\cdot   \\
0&\cA_2&\cB_3&\cC_3&0&\cdot\\
 \vdots & \vdots & \vdots & \vdots & \vdots
& \vdots
\end{bmatrix},\;
  \wt\cT_0=\begin{bmatrix}
\wt\cB_1 & \wt\cC_1 & 0 &0& 0&
\cdot  \\
 \wt\cA_1 & \wt\cB_2 & \wt\cC_2 &0& 0 &
\cdot   \\
0&\wt\cA_2&\wt\cB_3&\wt\cC_3&0&\cdot\\
 \vdots & \vdots & \vdots & \vdots & \vdots
& \vdots
\end{bmatrix},
\]
}
where $\cA_n,\cB_n,\cC_n$, and $\wt\cA_n,\wt\cB_n$, $\wt\cC_n$ are
given by \eqref{BLOKIT} and \eqref{BLOKIWT}. Since the matrices
$\cT_0$ and $\wt\cT_0$ are obtained from $\cU_0$ and $\wt\cU_0$ by
deleting the first rows and the first columns, we will call them
\textit{truncated block operator CMV matrices}. Observe that from
the definitions of $\cL_0$, $\cM_0$, $\wt\cM_0,$ $\cT_0$, and
$\wt\cT_0$ it follows that $\cT_0$ and $\wt\cT_0$ are products of
two block-diagonal matrices
\begin{multline}
\label{T0prod} \cT_0=\cT_0(\{\Gamma_n\}_{n\ge 0})
= \left(-\Gamma^*_0\oplus {\bf J}_{\Gamma_2}\oplus\ldots\oplus{\bf
J}_{\Gamma_{2n}}\oplus\ldots\right)\times\\
\qquad\times \left({\bf J}_{\Gamma_1}\oplus {\bf
J}_{\Gamma_3}\oplus\ldots\oplus{\bf
J}_{\Gamma_{2k-1}}\oplus\ldots\right)
\end{multline}
\begin{multline}
\label{WT0prod} \wt\cT_0=\wt\cT_0(\{\Gamma_k\}_{k\ge 0})
=\left({\bf J}_{\Gamma_1}\oplus {\bf
J}_{\Gamma_3}\oplus\ldots\oplus{\bf
J}_{\Gamma_{2k-1}}\oplus\ldots\right)\times\\
\qquad\times\left(-\Gamma^*_0\oplus {\bf
J}_{\Gamma_2}\oplus\ldots\oplus{\bf
J}_{\Gamma_{2k}}\oplus\ldots\right).
\end{multline}
In particular, it follows that
 $\left(\cT_0(\{\Gamma_k\}_{k\ge
0})\right)^*=\wt\cT_0(\{\Gamma^*_k\}_{k\ge 0}).$
From \eqref{T0prod} and \eqref{WT0prod} we have
$\cV_0\cT_0=\wt\cT_0\cV_0,$
where the unitary operator $\cV_0$ is defined by \eqref{V0}.
Therefore, the operators $\cT_0$ and $\wt\cT_0$ are unitarily
equivalent, in particular the following equalities
\begin{equation}
\label{uneq}
\begin{array}{l} \wt \cB_k{\bf J}_{\Gamma_{2k-1}}={\bf
J}_{\Gamma_{2k-1}}\cB_k,\; \wt \cA_k{\bf J}_{\Gamma_{2k-1}}
={\bf
J}_{\Gamma_{2k+1}}\cA_k,\\
 \wt \cC_k{\bf J}_{\Gamma_{2k+1}}={\bf
J}_{\Gamma_{2k-1}}\cC_k,\; k\ge 1,
\end{array}
\end{equation}
hold true.
\begin{proposition}
\label{COIUN}\cite{ArlMFAT2009}. Let $\Theta\in{\bf S}(\sM,\sN)$ and
let $\{\Gamma_k\}_{k\ge 0}$ be the Schur parameters of $\Theta$.
Suppose $\Gamma_k$ is neither isometric nor co-isometric for each
$k$. Let the function $\Omega\in{\bf S}(\sK,\sL)$ coincides with
$\Theta$ in the sense of  \cite{SF} and let $\{G_k\}_{k\ge 0}$ be
the Schur parameters of $\Omega$. Then truncated block operator CMV
matrices $\cT_0(\{\Gamma_k\}_{k\ge 0})$ and $\cT_0(\{G_k\}_{k\ge
0})$ (respect., $\wt\cT_0(\{\Gamma_k\}_{k\ge 0})$ and
$\wt\cT_0(\{G_k\}_{k\ge 0})$) are unitarily equivalent.
\end{proposition}

\subsection{Simple conservative realizations of the Schur class
function by means of its Schur parameters}\label{RRR}
 Set
\[
\begin{array}{l}
\cG_0=\cG_0(\{\Gamma_k\}_{k\ge}
0):=\begin{bmatrix}D_{\Gamma^*_0}\Gamma_1&D_{\Gamma^*_0}D_{\Gamma^*_1}&0&0&\ldots\end{bmatrix}\\
\qquad=D_{\Gamma^*_0}\begin{bmatrix}
\Gamma_1&D_{\Gamma^*_1}&0&\ldots\end{bmatrix}\in\bL(
\sH_0,\sN),\\
\wt\cG_0=\wt\cG_0(\{\Gamma_k\}_{k\ge
0}):=\begin{bmatrix}D_{\Gamma^*_0}&0&\ldots\end{bmatrix}=D_{\Gamma^*_0}
\begin{bmatrix} I_\sN&0&\ldots\end{bmatrix}\in\bL(\wt\sH_0,\sN),\\
\cF_0=\cF_0(\{\Gamma_k\}_{k\ge 0}):=\begin{bmatrix} D_{\Gamma_0}\cr
0\cr\vdots
\end{bmatrix}
\in\bL(\sM,\sH_0),\;
\wt\cF_0=\wt\cF_0(\{\Gamma_k\}_{k\ge 0}):=\begin{bmatrix}\Gamma_1
D_{\Gamma_0}\cr D_{\Gamma_1}D_{\Gamma_0}\cr0\cr\vdots\end{bmatrix}
\in\bL(\sM,\wt\sH_0).
\end{array}
\]
Then the operators $\cU_0$ and $\wt\cU_0$ defined by \eqref{defcmv}
and taking the form \eqref{U0} and \eqref{WTU0} can be represented
by $2\times 2$ block operator matrices
\[
\begin{array}{l}
\cU_0=\begin{bmatrix} \Gamma_0&\cG_0\cr \cF_0
&\cT_0\end{bmatrix}:\begin{array}{l}\sM\\\oplus\\\sH_0\end{array}\to
\begin{array}{l}\sN\\\oplus\\\sH_0\end{array},
\;\wt\cU_0=\begin{bmatrix} \Gamma_0&\wt\cG_0\cr\wt\cF_0
&\wt\cT_0\end{bmatrix}:\begin{array}{l}\sM\\\oplus\\\wt
\sH_0\end{array}\to
\begin{array}{l}\sN\\\oplus\\\wt\sH_0\end{array}.
\end{array}
\]
Recall \cite{A} that the discrete time-invariant system
\[
\Sigma=\left\{\begin{bmatrix} D&C\cr B
&A\end{bmatrix};\sM,\sN,\cH\right\}
\]
is called conservative if the operator
$$U=\begin{bmatrix}
D&C\cr B
&A\end{bmatrix}:\begin{array}{l}\sM\\\oplus\\\cH\end{array}\to
\begin{array}{l}\sN\\\oplus\\\cH\end{array}$$
is unitary. The function
\[
\Theta_{\Sigma}(z)=D+zC(I_\cH-zA)^{-1}B
\]
is called the transfer function of the system $\Sigma$ \cite{A}.
 Define the
following conservative systems:
\begin{equation}
\label{CMVMODEL} \begin{array}{l} \zeta_0=\left\{\begin{bmatrix}
\Gamma_0&\cG_0\cr \cF_0
&\cT_0\end{bmatrix};\sM,\sN,\sH_0\right\}=\left\{\cU_0(\{\Gamma_k\}_{k\ge 0});\sM,\sN,\sH_0(\{\Gamma_n\}_{k\ge 0})
\right\},\\
\wt\zeta_0=\left\{\begin{bmatrix} \Gamma_0&\wt\cG_0\cr \wt\cF_0
&\wt\cT_0\end{bmatrix};\sM,\sN,\wt\sH_0\right\}=\left\{\wt\cU_0(\{\Gamma_k\}_{k\ge
0});\sM,\sN,\wt\sH_0(\{\Gamma_k\}_{k\ge 0})\right\}.
\end{array}
\end{equation}
Equalities \eqref{MOV} and \eqref{MU} yield that systems $\zeta_0$
and $\wt\zeta_0$ are unitarily equivalent.  Hence, $\zeta_0$ and
$\wt\zeta_0$ have equal transfer functions \cite{A}. Let $\Theta(z)$
be the transfer function of $\zeta_0$ ($\wt \zeta_0$). Then
\[
\Theta(z)=\Gamma_0+z\cG_0(I_\sH-z\cT_0)^{-1}\cF_0=\Gamma_0+z\wt\cG_0(I_{\wt
\sH}-z\wt \cT_0)^{-1}\wt\cF_0.
\]
Using expressions for CMV matrices $\cU_0$ and $\wt\cU_0$ we obtain
\begin{equation}
\label{charfunc1}
\Theta(z)=\Gamma_0+zD_{\Gamma^*_0}\begin{bmatrix}\Gamma_1&D_{\Gamma^*_1}\end{bmatrix}
\left(P_{\cH_1}\left(I_\sH-z\cT_0\right)^{-1}\uphar\sD_{\Gamma_0}\right)D_{\Gamma_0}
\end{equation}
\begin{equation}
\label{charfunc2} \Theta(z)=\Gamma_0+zD_{\Gamma^*_0}
\left(P_{\sD_{\Gamma^*_0}}\left(I_{\wt \sH}-z\wt
\cT_0\right)^{-1}\uphar\wt\cH_1\right)\begin{bmatrix}\Gamma_1\cr
D_{\Gamma_1}\end{bmatrix}D_{\Gamma_0}.
\end{equation}
The next theorem has been established in \cite{ArlMFAT2009}, using
conservative realizations of the Schur algorithm obtained in
\cite{OAM2009}.

\begin{theorem}
\label{cmvmod}\cite{ArlMFAT2009}. 1) The unitarily equivalent
conservative systems $\zeta_0$ and $\wt\zeta_0$ given by
\eqref{CMVMODEL} are simple and the Schur parameters of the transfer
function $\Theta$ of $\zeta_0$ and $\wt\zeta_0$ are
$\{\Gamma_n\}_{n\ge 0}.$

2) Let $\Theta\in {\bf S}(\sM,\sN)$ and let $\{\Gamma_n\}_{n\ge 0}$
be the Schur parameters of $\Theta$. Then the systems
\eqref{CMVMODEL} are simple conservative realizations of $\Theta$.
\end{theorem}

\section{Connections between $\Theta$ and $\Theta_k$} \label{svyasi}
In the sequel we need some sub-matrices of block operator CMV matrices, which will play essential role
 in the parametrization of all solutions to the Schur problem.

 Suppose
\[
\Gamma_0\in\bL(\sM,\sN),\;\Gamma_1\in\bL(\sD_{\Gamma_0},\sD_{\Gamma^*_0}),\ldots
\]
 is a choice sequence and
\begin{equation}
\label{nevir} \sD_{\Gamma_{2n+1}}\ne \{0\},\;
\sD_{\Gamma^*_{2n+1}}\ne\{0\}
\end{equation}
for some $n$.
 Define the Hilbert spaces
\begin{equation}
\label{HSN}
\begin{array}{l}
\cH_k:=\begin{array}{l}\sD_{\Gamma_{2k-2}}\\\oplus\\\sD_{\Gamma^*_{2k-1}}\end{array},\;
\wt
\cH_k:=\begin{array}{l}\sD_{\Gamma^*_{2k-2}}\\\oplus\\\sD_{\Gamma_{2k-1}}\end{array},\;
k=1,\ldots, n+1, \\
\cK_n=\bigoplus\limits_{k=1}^{n+1}\cH_k,\;
\wt\cK_n=\bigoplus\limits_{k=1}^{n+1}\wt\cH_k.
\end{array}
\end{equation}

\subsection{Sub-matrices $\cS_{n}$ and $\wt \cS_n$ of block operator CMV matrices}

\begin{proposition}
\label{contrunc} Let \eqref{nevir} holds true. Then the operators
$\cS_n$ and $\wt\cS_n$ given by three-diagonal block operators
matrices
\begin{equation}
\label{submtr1} \cS_n=\cS_n\left(\Gamma_0,
\Gamma_1,\ldots,\Gamma_{2n+2}\right):=
  \begin{bmatrix}
\cB_1 & \cC_1 & 0 &0& \cdot &
\cdot&0  \\
\cA_1 & \cB_2 & \cC_2 &0& \cdot &
\cdot& 0  \\
\vdots & \vdots & \vdots & \vdots & \vdots & \vdots & \vdots\\
\cdot&\cdot&\cdot&\cdot&\cdot&\cdot&\cB_{n+1}
\end{bmatrix},
\end{equation}
\begin{equation}
\label{submtr2} \wt \cS_n=\wt \cS_n\left(\Gamma_0,
\Gamma_1,\ldots,\Gamma_{2n+2}\right):=
  \begin{bmatrix}
\wt\cB_1 & \wt \cC_1 & 0 &0& \cdot &
\cdot&0  \\
\wt\cA_1 & \wt\cB_2 & \wt \cC_2 &0& \cdot &
\cdot&0\\
\vdots & \vdots & \vdots & \vdots & \vdots & \vdots & \vdots\\
\cdot&\cdot&\cdot&\cdot&\cdot&\cdot&\wt \cB_{n+1}
\end{bmatrix}
\end{equation}
and acting in the Hilbert spaces $\cK_n$ and $\wt \cK_n$,
respectively, are unitarily equivalent contractions and
\begin{equation}
\label{sopr} \left(\wt \cS_n\left(\Gamma_0,
\Gamma_1,\ldots,\Gamma_{2n+2}\right)\right)^*=
\cS_n\left(\Gamma^*_0, \Gamma^*_1,\ldots,\Gamma^*_{2n+2}\right).
\end{equation}
\end{proposition}
\begin{proof}
Since $ \cS_n=P_{\cK_n}\cT_0\uphar\cK_n,$ $\wt \cS_n=P_{\wt
\cK_n}\wt\cT_0\uphar\wt\cK_n, $ the operators $\cS_n$ and $\wt
\cS_n$ are contractions and can be represented as products
\begin{equation}
\label{prod1}\cS_n=\cW_{n}\cV_n,\; \wt \cS_n=\cV_n\cW_{n},
\end{equation}
where
\begin{equation} \label{VN}
\cV_n=\bigoplus\limits_{k=1}^{n+1}{\bf
J}_{\Gamma_{2k-1}}:\cK_{n}\to\wt\cK_{n},
\end{equation}
\begin{equation}
\label{V01} \cW_{n}=-\Gamma^*_0\oplus\bigoplus\limits_{k=1}^{n}{\bf
J}_{\Gamma_{2k}}\oplus\Gamma_{2n+2}:\wt\cK_{n}\to\cK_{n}.
\end{equation}
In \eqref{V01} it is convenient to represent $\wt\cK_n$ as
\[
\wt\cK_n=\sD_{\Gamma^*_0}\oplus\bigoplus
\limits_{k=1}^n\begin{array}{l}\sD_{\Gamma_{2k-1}}\\\oplus\\\sD_{\Gamma^*_{2k}}\end{array}
\oplus\sD_{\Gamma_{2n+1}}.
\]
 Then from \eqref{uneq} follows the
equality $\cV_n \cS_n=\wt \cS_n\cV_n.$ Since $\cV_n$ unitarily maps
$\cK_n$ onto $\wt\cK_n$, the operators $\cS_n$ and $\wt \cS_n$ are
unitarily equivalent. Relation \eqref{sopr} follows from
\eqref{BLOKIT} and \eqref{BLOKIWT}.
\end{proof}

\subsection{The matrix $\cS_{n,0}$}
It should be mentioned that
\begin{enumerate}
\item
in the matrices $\cS_{n}$ and $\wt \cS_{n}$, the operator
$\Gamma_{2n+2}$ is contained only in the entries $\cB_{n+1}$ and
$\wt \cB_{n+1}$ (see \eqref{BLOKIT}, \eqref{BLOKIWT}):
\[
\begin{array}{l}
\cB_{n+1}=\begin{bmatrix}-\Gamma^*_{2n}\Gamma_{2n+1}&-\Gamma^*_{2n}D_{\Gamma^*_{2n+1}}\cr
\Gamma_{2n+2}D_{\Gamma_{2n+1}}&-\Gamma_{2n+2}\Gamma^*_{2n+1}\end{bmatrix}\in\bL(\cH_{n+1}),\\
\wt\cB_{n+1}=\begin{bmatrix}
-\Gamma_{2n+1}\Gamma^*_{2n}&D_{\Gamma^*_{2n+1}}\Gamma_{2n+2}\cr
-D_{\Gamma_{2n+1}}\Gamma^*_{2n}&-\Gamma^*_{2n+1}\Gamma_{2n+2}\end{bmatrix}\in\bL(\wt\cH_{n+1});
\end{array}
\]
\item the entries $\cB_{n+1}$ and $\wt\cB_{n+1}$ admits the
representations
\begin{multline*}
\cB_{n+1}=\begin{bmatrix}-\Gamma^*_{2n}\Gamma_{2n+1}&-\Gamma^*_{2n}D_{\Gamma^*_{2n+1}}\cr
\Gamma_{2n+2}D_{\Gamma_{2n+1}}&-\Gamma_{2n+2}\Gamma^*_{2n+1}\end{bmatrix}=
\begin{bmatrix}-\Gamma^*_{2n}\Gamma_{2n+1}&-\Gamma^*_{2n}D_{\Gamma^*_{2n+1}}\cr
0 & 0\end{bmatrix}\\
\qquad\qquad\quad +\begin{bmatrix}0&0\cr
\Gamma_{2n+2}D_{\Gamma_{2n+1}}&-\Gamma_{2n+2}\Gamma^*_{2n+1}\end{bmatrix}\\
=\begin{bmatrix}-\Gamma^*_{2n}\Gamma_{2n+1}&-\Gamma^*_{2n}D_{\Gamma^*_{2n+1}}\cr
0&0\end{bmatrix}+\begin{bmatrix}0&0\cr
0&\Gamma_{2n+2}\end{bmatrix}{\bf J}_{\Gamma_{2n+1}},\\
\wt \cB_{n+1}=\begin{bmatrix}
-\Gamma_{2n+1}\Gamma^*_{2n}&D_{\Gamma^*_{2n+1}}\Gamma_{2n+2}\cr
-D_{\Gamma_{2n+1}}\Gamma^*_{2n}&-\Gamma^*_{2n+1}\Gamma_{2n+2}\end{bmatrix}=
\begin{bmatrix}
-\Gamma_{2n+1}\Gamma^*_{2n}& 0\cr -D_{\Gamma_{2n+1}}\Gamma^*_{2n}&
0\end{bmatrix}
+\begin{bmatrix} 0&D_{\Gamma^*_{2n+1}}\Gamma_{2n+2}\cr
0&-\Gamma^*_{2n+1}\Gamma_{2n+2}\end{bmatrix}\\
=\begin{bmatrix} -\Gamma_{2n+1}\Gamma^*_{2n}& 0\cr
-D_{\Gamma_{2n+1}}\Gamma^*_{2n}& 0\end{bmatrix}+{\bf
J}_{\Gamma_{2n+1}}\begin{bmatrix}0&0\cr
0&\Gamma_{2n+2}\end{bmatrix}.
\end{multline*}
\end{enumerate}
\begin{proposition}
\label{new1}Let
$$\Gamma_0\in\bL(\sM,\sN),\;
\Gamma_1\in\bL(\sD_{\Gamma_0},\sD_{\Gamma^*_0}),\ldots,\Gamma_{2n+1}\in\bL(\sD_{\Gamma_{2n}},\sD_{\Gamma^*_{2n}})$$
be a finite choice sequence consisting of neither isometric nor
co-isometric operators. Then for each contraction
$\Gamma\in\bL(\sD_{\Gamma_{2n+1}},\sD_{\Gamma^*_{2n+1}})$ the
operators
\[
\cS_{n,\Gamma}:=\cS_n\left(\Gamma_0, \Gamma_1, \ldots,\Gamma_{2n+1},
\Gamma\right),\; \wt \cS_{n,\Gamma}:= \wt \cS_n \left(\Gamma_0,
\Gamma_1, \ldots,\Gamma_{2n+1}, \Gamma\right),
\]
corresponding
to the choice sequence
$$\Gamma_0,\;
\Gamma_1,\ldots,\Gamma_{2n+1},\; \Gamma,
$$
 are unitarily equivalent contractions.
\end{proposition}
\begin{proof} Take an arbitrary
$S\in\bS(\sD_{\Gamma_{2n+1}},\sD_{\Gamma^*_{2n+1}})$ with
$S(0)=\Gamma$. If
$$\gamma_0=\Gamma,\; \gamma_1,\gamma_2,\ldots$$
are the Schur parameters of $S$, then
$$ \Gamma_0,\Gamma_1,\ldots,\Gamma_{2n+1},\Gamma,\gamma_1,\ldots$$
is the choice sequence. Applying Proposition \ref{contrunc} we
arrive at the statement of the proposition.
\end{proof}

Let the finite choice sequence
$$\Gamma_0\in\bL(\sM,\sN),\;
\Gamma_1\in\bL(\sD_{\Gamma_0},\sD_{\Gamma^*_0}),\ldots,\Gamma_{2n+1}\in\bL(\sD_{\Gamma_{2n}},\sD_{\Gamma^*_{2n}})$$
be given and let
$\Gamma\in\bL(\sD_{\Gamma_{2n+1}},\sD_{\Gamma^*_{2n+1}})$ be a
contraction.
 Set
\[
\begin{array}{l}
\cB_{n+1,0}:=\begin{bmatrix}-\Gamma^*_{2n}\Gamma_{2n+1}&-\Gamma^*_{2n}D_{\Gamma^*_{2n+1}}\cr
0&0\end{bmatrix}\in\bL(\cH_{n+1}),\\
\wt\cB_{n+1,0}:=\begin{bmatrix} -\Gamma_{2n+1}\Gamma^*_{2n}& 0\cr
-D_{\Gamma_{2n+1}}\Gamma^*_{2n}& 0\end{bmatrix}\in\bL(\wt\cH_{n+1})
\end{array}
\]

 \begin{equation}
\label{tgeo} \cS_{n,0}:=\cS_{n}\left(\Gamma_0, \Gamma_1,
\ldots,\Gamma_{2n+1}, 0 \right),\; \wt \cS_{n,0}:=\wt
\cS_n\left(\Gamma_0, \Gamma_1, \ldots,\Gamma_{2n+1}, 0\right),
\end{equation}
where $0\in\bL(\sD_{\Gamma_{2n+1}},\sD_{\Gamma^*_{2n+1}}).$
 Then
\[
\cS_{n,0}=
  \begin{bmatrix}
\cB_1 & \cC_1 & 0 &0& \cdot & \cdot&0  \cr \cA_1 & \cB_2 & \cC_2 &0&
\cdot & \cdot& 0  \cr \vdots & \vdots & \vdots & \vdots & \vdots &
\vdots & \vdots\cr \cdot&\cdot&\cdot&\cdot&\cdot&\cB_{n}&\cC_{n}\cr
\cdot&\cdot&\cdot&\cdot&\cdot&\cA_{n}&\cB_{n+1,0}
\end{bmatrix},\;\wt \cS_{n,0}=
  \begin{bmatrix}
\wt\cB_1 & \wt \cC_1 & 0 &0& \cdot & \cdot&0  \cr \wt\cA_1 &
\wt\cB_2 & \wt \cC_2 &0& \cdot & \cdot&0\cr \vdots & \vdots & \vdots
& \vdots & \vdots & \vdots & \vdots\cr
\cdot&\cdot&\cdot&\cdot&\cdot&\wt\cB_n&\wt \cC_{n}\cr
\cdot&\cdot&\cdot&\cdot&\cdot&\wt\cA_n&\wt \cB_{n+1,0}
\end{bmatrix},
\]
\begin{equation}
\label{tn}
\cS_{n,\Gamma}=\cS_{n,0}+j_{n+1}\Gamma\begin{bmatrix}D_{\Gamma_{2n+1}}&-\Gamma^*_{2n+1}
\end{bmatrix}P_{\cH_{n+1}},
\end{equation}
\begin{equation}
\label{wtn} \wt \cS_{n,\Gamma}=\wt \cS_{n,0}+\wt
j_{n+1}\begin{bmatrix}D_{\Gamma^*_{2n+1}}\cr-\Gamma^*_{2n+1}
\end{bmatrix}\Gamma P_{\sD_{\Gamma_{2n+1}}},
\end{equation}
where $j_{n+1}$ is the embedding operator from
$\sD_{\Gamma^*_{2n+1}}$ into $\cK_{n}$, $\wt j_{n+1}$ is the
embedding operator from $\wt\cH_{n+1}$ into $\wt \cK_{n}$. Observe
that the block operator matrices $\cS_{n,0}$ and $\wt \cS_{n,0}$ can
be obtained from truncated CMV matrices $\cT_0$ and $\wt\cT_0$
corresponding to the infinite choice sequence
\[
\Gamma_0,\Gamma_1,\ldots,\Gamma_{2n},\Gamma_{2n+1},0,0,\ldots,
\]
where $0\in\bL(\sD_{\Gamma_{2n+1}},\sD_{\Gamma^*_{2n+1}})$. Let
\begin{equation}
\label{wno}
\cW_{n,0}=-\Gamma^*_0\oplus\bigoplus\limits_{k=1}^{n}{\bf
J}_{\Gamma_{2k}}\oplus 0:\wt\cK_n\to\cK_n.
\end{equation}
Due to \eqref{V01}, \eqref{wno}, and \eqref{prod1}, the operators
$\cS_{n,0}$ and $\wt \cS_{n,0}$ admit factorizations
\begin{equation}
\label{prdtn}
\begin{array}{l}
\cS_{n,0}=\cW_{n,0}\cV_n=\left(-\Gamma^*_0\oplus\bigoplus\limits_{k=1}^{n}{\bf
J}_{\Gamma_{2k}}\oplus 0\right)\times \left(
\bigoplus\limits_{k=1}^{n+1}{\bf J}_{\Gamma_{2k-1}}\right),\\
\wt \cS_{n,0}=\cV_n\cW_{n,0}= \left(
\bigoplus\limits_{k=1}^{n+1}{\bf J}_{\Gamma_{2k-1}}\right)\times
\left(-\Gamma^*_0\oplus\bigoplus\limits_{k=1}^{n}{\bf
J}_{\Gamma_{2k}}\oplus 0\right).
\end{array}
\end{equation}
Notice that
\[
P_{\sD_{\Gamma^*_{2n+1}}}\cS_{n,0}=0,\; \wt
\cS_{n,0}\uphar\sD_{\Gamma_{2n+1}}=0.
\]
 Our next goal is to express for $z\in\dD$: \begin{enumerate}
\item the resolvent
$\left(I_{\cK_{n}}-z\cS_{n,\Gamma}\right)^{-1}$ through the
resolvent $\left(I_{\cK_{n}}-z\cS_{n,0}\right)^{-1}$,
\item
 the resolvent
$\left(I_{\wt \cK_{n}}-z\wt \cS_{n,\Gamma}\right)^{-1}$ through
$\left(I_{\wt \cK_{n}}-z\wt \cS_{n,0}\right)^{-1}$.
\end{enumerate}
\begin{proposition} If $|z|<1$, then
\label{resform}
\begin{multline*}
\left(I_{\cK_{n}}-z\cS_{n,\Gamma}\right)^{-1}=\left(I_{\cK_{n}}-z\cS_{n,0}\right)^{-1}
+z\left(\left(I_{\cK_{n}}-z\cS_{n,0}\right)^{-1}\uphar\sD_{\Gamma^*_{2n+1}}\right)\times\\
\left(I_{\sD_{\Gamma^*_{2n+1}}}-
z\Gamma\begin{bmatrix}D_{\Gamma_{2n+1}}&-\Gamma^*_{2n+1}\end{bmatrix}
\left(P_{\cH_{n+1}}\left(I_{\cK_{n}}-z\cS_{n,0}\right)^{-1}\uphar\sD_{\Gamma^*_{2n+1}}\right)\right)^{-1}\times\\
\Gamma\begin{bmatrix}D_{\Gamma_{2n+1}}&-\Gamma^*_{2n+1}\end{bmatrix}
P_{\cH_{n+1}}\left(I_{\cK_{n}}-z\cS_{n,0}\right)^{-1},
\end{multline*}
\begin{multline*}
\left(I_{\wt\cK_{n}}-z\wt \cS_{n,\Gamma}\right)^{-1}=\left(I_{\wt\cK_{n}}-z\wt \cS_{n,0}\right)^{-1}\\
+z\left(\left(I_{\wt\cK_{n}}-z\wt
\cS_{n,0}\right)^{-1}\uphar\wt\cH_{n+1}\right)\begin{bmatrix}D_{\Gamma^*_{2n+1}}\cr-\Gamma^*_{2n+1}
\end{bmatrix}\Gamma\times\\
\left(I_{\sD_{\Gamma_{2n+1}}}-z\left(P_{\sD_{\Gamma_{2n+1}}}
\left(I_{\wt\cK_{n}}-z\wt
\cS_{n,0}\right)^{-1}\uphar\wt\cH_{n+1}\right)
\begin{bmatrix}D_{\Gamma^*_{2n+1}}\cr-\Gamma^*_{2n+1}
\end{bmatrix}\Gamma\right)^{-1}\\
 P_{
\sD_{\Gamma_{2n+1}}}\left(I_{\wt\cK_{n}}-z\wt \cS_{n,0}\right)^{-1}.
\end{multline*}
\end{proposition}
\begin{proof}
Let $A$ and $B$ be bounded operators in the Hilbert space $H$.
Suppose that $-1\in\rho(A)\cap\rho(B)$. Denote $S:=A-B$, i.e.,
\[
A=B+S.
\]
Then
\[
I+A=(I+B)(I+(I+B)^{-1}S),
\]
where $I=I_H$. It follows that
$\left(I+(I+B)^{-1}S\right)^{-1}\in\bL(H)$ and
\[
(I+A)^{-1}=\left(I+(I+B)^{-1}S\right)^{-1}(I+B)^{-1}.
\]
On the other hand
\[
(I+A)^{-1}-(I+B)^{-1}=-(I+B)^{-1}S(I+A)^{-1}.
\]
Hence
\[
(I+A)^{-1}=(I+B)^{-1}-(I+B)^{-1}S\left(I+(I+B)^{-1}S\right)^{-1}(I+B)^{-1}.
\]
Similarly
\[
(I+A)^{-1}=(I+B)^{-1}-(I+B)^{-1}\left(I+S(I+B)^{-1}\right)^{-1}S(I+B)^{-1}.
\]
 If $\sN$ is a proper subspace in $H$ and $\ran
S\subseteq\sN$, then 
\[
(I+A)^{-1}=(I+B)^{-1}-(I+B)^{-1}\left(I_\sN+S(I+B)^{-1}\uphar\sN\right)
^{-1}S(I+B)^{-1}
\]
If $\ker S\supseteq H\ominus \sN$, then $S=SP_\sN$ and
\[
(I+A)^{-1}=(I+B)^{-1}-(I+B)^{-1}S\left(I_\sN+P_\sN(I+B)^{-1}
S\right)^{-1}P_{\sN}(I+B)^{-1}.
\]

Applying the latter equalities to $A=-z\cS_{n,\Gamma}$ and $A=-z\wt
\cS_{n,\Gamma}$, $B=-z\cS_{n,0}$ and $B=-z\wt \cS_{n,0}$ and using
\eqref{tn} and \eqref{wtn} we get formulas in the proposition.
\end{proof}

As it is well-known, the resolvent $R_T(\lambda)=(T-\lambda I)^{-1}$  of a block operator matrix  $
T=\begin{pmatrix} B&D\cr
C&A\end{pmatrix}:\begin{array}{l}\sN\\\oplus\\\cH\end{array}\to
\begin{array}{l}\sN\\\oplus\\\cH\end{array}
$
takes the form (the Schur-Frobenius formula)
\begin{equation}
\label{Sh-Frob}
\begin{array}{l}
R_T(\lambda)=
\begin{pmatrix}-V^{-1}(\lambda)&V^{-1}(\lambda)DR_A(\lambda)\cr
R_A(\lambda)CV^{-1}(\lambda)&R_A(\lambda)\left(I_\cH-CV^{-1}(\lambda)DR_A(\lambda)\right)
\end{pmatrix}\\
\qquad\mbox{for}\quad \lambda\in\rho(T)\cap\rho(A),
\end{array}
\end{equation}
where
 $$V(\lambda):=\lambda I_\sN-B+DR_A(\lambda)C,\; \lambda\in\rho(A).$$
 and $I=I_{\sN\oplus\cH}$.
It follows that
\[
P_\sN R_T(\lambda)\uphar\sN=-(\lambda I_\sN-B+DR_A(\lambda)C)^{-1}
\]
and
\[
zP_\sN(I-z T)^{-1}\uphar\sN=\left(I_\sN-z\left(B+zD(I_\cH-z
A)^{-1}C\right)\right)^{-1}.
\]
The next statement plays an essential role in the sequel.

\begin{theorem}
 \label{formula1}
Let $\Theta\in{\bS}(\sM,\sN)$ and let $\{\Gamma_k\}_{k\ge 0}$ be its
Schur parameters. Suppose that $\sD_{\Gamma_{2n+1}}\ne \{0\}$ and
$\sD_{\Gamma^*_{2n+1}}\ne \{0\}$. Then for each $z\in\dD$ one has
\begin{equation}
\label{schfr1} P_{\cK_n}\left(I_\sH-z\cT_0\right)^{-1}\uphar\cK_n=
\left(I_{\cK_n}-z\cS_{n,\Theta_{2n+2}(z)}\right)^{-1}\uphar\cK_n,
\end{equation}
\begin{equation}
\label{schfr2}
P_{\wt\cK_n}\left(I_{\wt\sH}-z\wt\cT_0\right)^{-1}\uphar\wt\cK_n=\left(I_{\wt\cK_n}-z\wt
\cS_{n,\Theta_{2n+2}(z)}\right)^{-1}\uphar\wt\cK_n,
\end{equation}
where $\Theta_{2n+2}$ is the function associated with $\Theta$ in
accordance with the Schur algorithm.
\end{theorem}
\begin{proof}
The CMV matrices $\cT_0$ and $\wt \cT_0$ can be represented as
follows
\begin{equation}
\label{matrrep} \cT_0=\begin{bmatrix}\cS_n&\cQ_n\cr\cD_n& \cT'
\end{bmatrix}:\begin{array}{l}\cK_n\\
\oplus\\\sH'\end{array}\to\begin{array}{l}\cK_n\\
\oplus\\\sH'\end{array},\;\wt\cT_0=\begin{bmatrix}\wt
\cS_n&\wt\cQ_n\cr\wt\cD_n& \wt\cT'
\end{bmatrix}:\begin{array}{l}\wt\cK_n\\
\oplus\\\wt\sH'\end{array}\to\begin{array}{l}\wt\cK_n\\
\oplus\\\wt\sH'\end{array},
\end{equation}
where
\[
\begin{array}{l}\sH'=\sH'\left(\{\Gamma_k\}_{k\ge
2n+2}\right)= \sH_0\ominus\cK_n,\\
\wt\sH'= \wt\sH'\left(\{\Gamma_k\}_{k\ge 2n+2}\right)= \wt
\sH_0\ominus\wt \cK_n,
\end{array}
\]
and
\[
 \cT'=\cT_0(\{\Gamma_{k}\}_{k\ge2n+2}),\;
\wt\cT'=\wt\cT_0(\{\Gamma_{k}\}_{k\ge2n+2})
\]
 are truncated CMV
matrices corresponding to the CMV matrices
$$\cU_{2n+2}=\cU_{2n+2}\left(\{\Gamma_{k}\}_{k\ge2n+2}\right)\;\mbox{and}\quad
\wt \cU_{2n+2}=\wt
\cU_{2n+2}\left(\{\Gamma_{k}\}_{k\ge2n+2}\right).$$ In order to
prove this theorem it is necessary to consider the following cases
1) $\Gamma_{2n+2}$ and $\Gamma_{2n+3}$ are neither isometric nor
 co-isometric, 
2) $\Gamma_{2n+3}$ is isometric, 3) $\Gamma_{2n+3}$ is co-isometric,
4) $\Gamma_{2n+3}$ is unitary,
5) $\Gamma_{2n+2}$ is isometric,
6) $\Gamma_{2n+2}$ is co-isometric,
7) $\Gamma_{2n+2}$ is unitary.
We consider only the cases 1), 2), and 6) leaving the rest for a
reader.

{\bf The operators $\Gamma_{2n+2}$ and $\Gamma_{2n+3}$ are neither
isometric nor  co-isometric}. In this case the entries in
\eqref{matrrep} take the form
\[
\cD_n=\begin{bmatrix}0&0&\ldots&0& \cA_{n+1}\cr
 0&0&\ldots&0&0\cr\vdots&\vdots&\vdots&\vdots&\vdots
\end{bmatrix}:\cK_n\to\sH',\; \cQ_n=\begin{bmatrix}0&0&\ldots\cr
\vdots&\vdots&\vdots\cr 0&0&\ldots\cr  \cC_{n+1}&0&\ldots
\end{bmatrix}:\sH'\to\cK_n
\]
\[
\wt \cD_n=\begin{bmatrix}0&0&\ldots&0& \wt \cA_{n+1}\cr
 0&0&\ldots&0&0\cr\vdots&\vdots&\vdots&\vdots&\vdots
\end{bmatrix}:\wt \cK_n\to\wt\sH',\; \wt\cQ_n=\begin{bmatrix}0&0&\ldots\cr
\vdots&\vdots&\vdots\cr 0&0&\ldots\cr  \wt\cC_{n+1}&0&\ldots
\end{bmatrix}:\wt\sH'\to\wt\cK_n
\]

 Observe that from \eqref{charfunc1} follows the
equality
\begin{multline*}
\Theta_{2n+2}(z)=\Biggl(\Gamma_{2n+2}+zD_{\Gamma^*_{2n+2}}\begin{bmatrix}\Gamma_{2n+3}&D_{\Gamma^*_{2n+3}}
\end{bmatrix}\times\\
\left(P_{\cH_{n+2}}\left(I_\sH-z\cT'\right)^{-1}\uphar\sD_{\Gamma_{2n+2}}\right)
D_{\Gamma_{2n+2}}\Biggr)\uphar\sD_{\Gamma_{2n+1}}.
\end{multline*}

From the Schur -Frobenius formula \eqref{Sh-Frob} we have
\[
P_{\cK_n}\left(I_\sH-z\cT_0\right)^{-1}\uphar\cK_n=\left(I_{\cK_n}-z\left(\cS_n+z
\cQ_n\left(I_{\sH'}-z\cT'\right)^{-1}\cD_n\right)\right)^{-1}.
\]
Since
\[
\begin{array}{l}
\cC_{n+1}=\begin{bmatrix}0&0\cr
D_{\Gamma^*_{2n+2}}\Gamma_{2n+3}&D_{\Gamma^*_{2n+2}}D_{\Gamma^*_{2n+3}}\end{bmatrix}:
\begin{array}{l}\sD_{\Gamma_{2n+2}}\\\oplus\\\sD_{\Gamma^*_{2n+3}}\end{array}=\cH_{n+2}
\to\begin{array}{l}\sD_{\Gamma_{2n}}\\\oplus\\\sD_{\Gamma^*_{2n+1}}\end{array}=\cH_{n+1},\\
\cA_{n+1}=\begin{bmatrix}D_{\Gamma_{2n+2}}D_{\Gamma_{2n+1}}&
-D_{\Gamma_{2n+2}}\Gamma^*_{2n+1}\cr
0&0\end{bmatrix}:\begin{array}{l}\sD_{\Gamma_{2n}}\\\oplus\\\sD_{\Gamma^*_{2n+1}}\end{array}=\cH_{n+1}
\to\begin{array}{l}\sD_{\Gamma_{2n+2}}\\\oplus\\\sD_{\Gamma^*_{2n+3}}\end{array}=\cH_{n+2},
\end{array}
\]
we get
\begin{multline*}
\cQ_n\left(I_{\sH'}-z\cT'\right)^{-1}\cD_n\\
\qquad=
D_{\Gamma^*_{2n+2}}\begin{bmatrix}\Gamma_{2n+3}&D_{\Gamma^*_{2n+3}}\end{bmatrix}
\left(P_{\cH_{n+2}}\left(I_{\sH'}-z\cT'\right)^{-1}\uphar\sD_{\Gamma_{2n+2}}\right)\times\\
D_{\Gamma_{2n+2}}
\begin{bmatrix}D_{\Gamma_{2n+1}}
-\Gamma^*_{2n+1}\end{bmatrix}P_{\cH_{n+1}},
\end{multline*}
\begin{multline*}
\cS_n+z\cQ_n\left(I_{\sH'}-z\cT'\right)^{-1}\cD_n=\cS_{n,0}+
\Gamma_{2n+2}\begin{bmatrix}
D_{\Gamma_{2n+1}}&
-{\Gamma^*_{2n+1}}\end{bmatrix}P_{\cH_{n+1}}\\
+ zD_{\Gamma^*_{2n+2}}\begin{bmatrix}
\Gamma_{2n+3}&D_{\Gamma^*_{2n+3}}\end{bmatrix}
\left(P_{\cH_{n+2}}\left(I_{\sH'}-z\cT'\right)^{-1}\uphar\sD_{\Gamma_{2n+2}}\right)\times\\
\qquad D_{\Gamma_{2n+2}}
\begin{bmatrix}D_{\Gamma_{2n+1}}
-\Gamma^*_{2n+1}\end{bmatrix}P_{\cH_{n+1}}\\
\qquad\qquad=\cS_{n,0}+
\Theta_{2n+2}(z)\begin{bmatrix}D_{\Gamma_{2n+1}}&-{\Gamma^*_{2n+1}}
\end{bmatrix}P_{\cH_{n+1}}=\cS_{n,\Theta_{2n+2}(z)}.\\
\end{multline*}
Now using the Schur-Frobenius formula, we arrive at \eqref{schfr1}.
Similarly \eqref{schfr2} can be proved.

 {\bf The operator $\Gamma_{2n+3}$ is isometric}. In this case
$\sD_{\Gamma_{2n+3}}=0$. We will prove \eqref{schfr2}. One can see
that
\[
\wt\sH'=\sD_{\Gamma^*_{2n+2}}\oplus\sD_{\Gamma^*_{2n+3}}\oplus\sD_{\Gamma^*_{2n+3}}\oplus\cdots,
\]
and in the matrix representation \eqref{matrrep} the entries
$\wt\cD_n$, $\wt\cQ_n$, and $\wt\cT'$ take the form (see Appendix A)
\[
\wt \cD_n=\begin{bmatrix}0&0&\ldots&0&
\Gamma_{2n+3}D_{\Gamma_{2n+2}}\cr
 0&0&\ldots&0&0\cr\vdots&\vdots&\vdots&\vdots&\vdots
\end{bmatrix},\;
\wt\cQ_n=\begin{bmatrix}0&0&\ldots\cr \vdots&\vdots&\vdots\cr
0&0&\ldots\cr D_{\Gamma^*_{2n+1}}D_{\Gamma^*_{2n+2}}&0&\ldots \cr
-\Gamma^*_{2n+1}D_{\Gamma^*_{2n+2}}&0&\ldots
\end{bmatrix},
\]
\[
\wt\cT'=\wt\cT_0\left\{\Gamma_{k}\}_{k\ge
2n+2}\right)=\begin{bmatrix}
-\Gamma_{2n+3}\Gamma^*_{2n+2}&I_{\sD_{\Gamma^*_{2n+3}}}&0&0&0&\ldots\cr
0&0&I_{\sD_{\Gamma^*_{2n+3}}}&0&0&\ldots\cr
0&0&0&I_{\sD_{\Gamma^*_{2n+3}}}&0&\ldots\cr
\vdots&\vdots&\vdots&\vdots&\vdots&\vdots
\end{bmatrix}.
\]
We note that since $\Gamma_{2n+3}$ is isometric operator from
$\bL(\sD_{\Gamma_{2n+2}},\sD_{\Gamma^*_{2n+2}})$, the function
$\Theta_{2n+2}\in\bS(\sD_{\Gamma_{2n+1}},\sD_{\Gamma^*_{2n+1}})$ is
of the form
\[
\Theta_{2n+2}(z)=\Gamma_{2n+2}+zD_{\Gamma^*_{2n+2}}
\left(I_{\sD_{\Gamma^*_{2n+2}}}+z\Gamma_{2n+3}\Gamma^*_{2n+2}\right)^{-1}\Gamma_{2n+3}D_{\Gamma_{2n+2}}.
\]
The CMV matrix corresponding to
\[
\Gamma_{2n+2},\;\Gamma_{2n+3},\;
0\in\bL(0,\sD_{\Gamma^*_{2n+3}}),\ldots
\]
 is
\[
\wt\cU_0\left(\left\{\Gamma_{k}\}_{k\ge 2n+2}\right\}\right)=\begin{bmatrix}
\Gamma_{2n+2}&D_{\Gamma^*_{2n+2}}&0&0&0&0&\ldots\cr
\Gamma_{2n+3}D_{\Gamma_{2n+2}}&-\Gamma_{2n+3}\Gamma^*_{2n+2}&I_{\sD_{\Gamma^*_{2n+3}}}&0&0&0&\ldots\cr
0&0&0&I_{\sD_{\Gamma^*_{2n+3}}}&0&0&\ldots\cr
0&0&0&0&I_{\sD_{\Gamma^*_{2n+3}}}&0&\ldots\cr
\vdots&\vdots&\vdots&\vdots&\vdots&\vdots&\vdots
\end{bmatrix}.
\]
Hence, we get
\[
\Theta_{2n+2}(z)=\left(\Gamma_{2n+2}+zD_{\Gamma^*_{2n+2}}
\left(P_{\sD_{\Gamma^*_{2n+2}}}\left(I_{\wt\sH'}-z\wt\cT'\right)^{-1}\uphar\sD_{\Gamma^*_{2n+2}}\right)
\Gamma_{2n+3}D_{\Gamma_{2n+2}}\right)\uphar\sD_{\Gamma_{2n+1}},
\]
\begin{multline*}
\wt
\cS_n+z\wt\cQ_n\left(I_{\wt\sH'}-z\wt\cT'\right)^{-1}\wt\cD_n=\wt
\cS_{n,0}+
\begin{bmatrix}D_{\Gamma^*_{2n+1}}\cr
-{\Gamma^*_{2n+1}}\end{bmatrix}\Gamma_{2n+2}P_{\sD_{\Gamma_{2n+1}}}\\
+ z\begin{bmatrix}D_{\Gamma^*_{2n+1}}\cr
-{\Gamma^*_{2n+1}}\end{bmatrix}D_{\Gamma^*_{2n+2}}
\left(P_{\cD_{\Gamma^*_{2n+2}}}\left(I_{\wt\sH'}-z\wt\cT'\right)^{-1}\uphar\sD_{\Gamma^*_{2n+2}}\right)
\Gamma_{2n+3}D_{\Gamma_{2n+2}}P_{\sD_{\Gamma_{2n+1}}}\\
\qquad\qquad=\wt \cS_{n,0}+\begin{bmatrix}D_{\Gamma^*_{2n+1}}\cr
-{\Gamma^*_{2n+1}}\end{bmatrix}
\Theta_{2n+2}(z)P_{\sD_{\Gamma_{2n+1}}}=\wt
\cS_{n,\Theta_{2n+2}(z)}.
\end{multline*}
{\bf The operator $\Gamma_{2n+2}$ is co-isometric}. Now
$\Theta_{2n+2}(z)=\Gamma_{2n+2}$ for all $z\in\dD$,
\[
\sH'=\sD_{\Gamma_{2n+2}}\oplus\sD_{\Gamma_{2n+2}}\oplus\ldots,
\]
$\cQ_n=0:\sH'\to\cK_n$. Therefore,
$\cQ_n\left(I_{\sH'}-z\cT'\right)^{-1}\cD_n=0$ and
\[\begin{array}{l}
\cS_n+z\cQ_n\left(I_{\sH'}-z\cT'\right)^{-1}\cD_n=\cS_{n,0}+
\Gamma_{2n+2}\begin{bmatrix}
D_{\Gamma_{2n+1}}& -{\Gamma^*_{2n+1}}\end{bmatrix}P_{\cH_{n+1}}
=\cS_{n,\Theta_{2n+2}(z)}.
\end{array}
\]
\end{proof}
\subsection{Connection between $\Theta$ and
$\Theta_{2n+2}$}\label{conn11}

 Applying Proposition \ref{resform} for fixed $z\in\dD$ and
$\Gamma=\Theta_{2n+2}(z)$ we get
\begin{multline}
\label{expr11}
\left(I_{\cK_{n}}-z\cS_{n,\Theta_{2n+2}(z)}\right)^{-1}=\left(I_{\cK_{n}}-z\cS_{n,0}\right)^{-1}
+\\
z\left(\left(I_{\cK_{n}}-z\cS_{n,0}\right)^{-1}\uphar\sD_{\Gamma^*_{2n+1}}\right)\times\\
\left(I_{\sD_{\Gamma^*_{2n+1}}}-
z\Theta_{2n+2}(z)\begin{bmatrix}D_{\Gamma_{2n+1}}&-\Gamma^*_{2n+1}\end{bmatrix}
\left(P_{\cH_{n+1}}\left(I_{\cK_{n}}-z\cS_{n,0}\right)^{-1}\uphar\sD_{\Gamma^*_{2n+1}}\right)\right)^{-1}\\
\times\Theta_{2n+2}(z)\begin{bmatrix}D_{\Gamma_{2n+1}}&-\Gamma^*_{2n+1}\end{bmatrix}
P_{\cH_{n+1}}\left(I_{\cK_{n}}-z\cS_{n,0}\right)^{-1},
\end{multline}
\begin{multline}
\label{expr22}
\left(I_{\wt\cK_{n}}-z\wt \cS_{n,\Theta_{2n+2}(z)}\right)^{-1}=\left(I_{\wt\cK_{n}}-z\wt \cS_{n,0}\right)^{-1}\\
+z\left(\left(I_{\wt\cK_{n}}-z\wt
\cS_{n,0}\right)^{-1}\uphar\wt\cH_{n+1}\right)\begin{bmatrix}D_{\Gamma^*_{2n+1}}\cr-\Gamma^*_{2n+1}
\end{bmatrix}\Theta_{2n+2}(z)\times\\
\left(I_{\sD_{\Gamma_{2n+1}}}-z\left(P_{\sD_{\Gamma_{2n+1}}}
\left(I_{\wt\cK_{n}}-z\wt
\cS_{n,0}\right)^{-1}\uphar\wt\cH_{n+1}\right)
\begin{bmatrix}D_{\Gamma^*_{2n+1}}\cr-\Gamma^*_{2n+1}
\end{bmatrix}\Theta_{2n+2}(z)\right)^{-1}\times\\
P_{\sD_{\Gamma_{2n+1}}}\left(I_{\wt\cK_{n}}-z\wt
\cS_{n,0}\right)^{-1}.
\end{multline}
Define the following operator functions in $\dD$:
\begin{multline}
\label{functt1} \left\{
\begin{array}{l}
\Theta^{(0)}_n(z):=\Gamma_0+zD_{\Gamma^*_0}\begin{bmatrix}\Gamma_1&D_{\Gamma^*_1}\end{bmatrix}
\left(P_{\cH_1}\left(I_{\cK_n}-z\cS_{n,0}\right)^{-1}\uphar\sD_{\Gamma_0}\right)D_{\Gamma_0}\\
\qquad\qquad\qquad\qquad\in\bL(\sM,\sN),\\
 A_n(z):=z\begin{bmatrix}D_{\Gamma_{2n+1}}&-\Gamma^*_{2n+1}
\end{bmatrix}\left(P_{\cH_{n+1}}\left(I_{\cK_n}-z\cS_{n,0}\right)^{-1}\uphar\sD_{\Gamma^*_{2n+1}}\right)\\
\qquad\qquad\qquad\qquad\in\bL(\sD_{\Gamma^*_{2n+1}},\sD_{\Gamma_{2n+1}}), \\
B_n(z):=z\begin{bmatrix}D_{\Gamma_{2n+1}}&-\Gamma^*_{2n+1}
\end{bmatrix}\left(P_{\cH_{n+1}}\left(I_{\cK_n}-z\cS_{n,0}\right)^{-1}\uphar\sD_{\Gamma_{0}}\right)D_{\Gamma_0}\\
\qquad\qquad\qquad\qquad\in\bL(\sM,\sD_{\Gamma_{2n+1}}),\\
C_n(z):=zD_{\Gamma^*_0}\begin{bmatrix}\Gamma_1&D_{\Gamma^*_1}\end{bmatrix}
\left(P_{\cH_1}\left(I_{\cK_n}-z\cS_{n,0}\right)^{-1}\uphar\sD_{\Gamma^*_{2n+1}}\right)\\
\qquad\qquad\qquad\qquad\in\bL(\sD_{\Gamma^*_{2n+1}},\sN),
\end{array}
\right.
\end{multline}
\begin{equation}
\label{cq1} \cQ_n(z):=\begin{bmatrix}\Theta^{(0)}_n(z)& C_n(z)\cr
B_n(z)&A_n(z)\end{bmatrix}:\begin{array}{l}\sM\\\oplus\\\sD_{\Gamma^*_{2n+1}}\end{array}\to
\begin{array}{l}\sN\\\oplus\\\sD_{\Gamma_{2n+1}}\end{array},
\end{equation}
\begin{equation}
\label{functt11} \left\{
\begin{array}{l}
\wt\Theta^{(0)}_n(z):=\Gamma_0+zD_{\Gamma^*_0}
\left(P_{\sD_{\Gamma^*_0}}\left(I_{\wt\cK_n}-z\wt
\cS_{n,0}\right)^{-1}\uphar\cH_1\right)\begin{bmatrix}\Gamma_1\cr
D_{\Gamma_1}\end{bmatrix} D_{\Gamma_0}\\
\qquad\qquad\qquad\qquad\in\bL(\sM,\sN), \\
\wt A_n(z):=z\left(P_{\sD_{\Gamma_{2n+1}}}\left(I_{\wt\cK_n}-z\wt
\cS_{n,0}\right)^{-1}\uphar\wt\cH_{n+1}\right)\begin{bmatrix}D_{\Gamma^*_{2n+1}}\cr-\Gamma^*_{2n+1}
\end{bmatrix}\\
\qquad\qquad\qquad\qquad\in\bL(\sD_{\Gamma^*_{2n+1}},\sD_{\Gamma_{2n+1}}),\\
\wt B_n(z):=z\left(P_{\sD_{\Gamma_{2n+1}}}\left(I_{\wt\cK_n}-z\wt
\cS_{n,0}\right)^{-1}\uphar\wt\cH_1\right)\begin{bmatrix}\Gamma_1\cr
D_{\Gamma_1}
\end{bmatrix}D_{\Gamma_0}\\
\qquad\qquad\qquad\qquad
\in\bL(\sM,\sD_{\Gamma_{2n+1}}), \\
\wt C_n(z):=zD_{\Gamma^*_0}
\left(P_{\sD_{\Gamma^*_0}}\left(I_{\wt\cK_n}-z\wt
\cS_{n,0}\right)^{-1}\uphar\wt\cH_{n+1}\right)\begin{bmatrix}D_{\Gamma_{2n+1}}\cr
-\Gamma^*_{2n+1}\end{bmatrix}\\
\qquad\qquad\qquad\qquad \in\bL(\sD_{\Gamma^*_{2n+1}},\sN),
\end{array}\right.
\end{equation}
\begin{equation}
\label{cq2} \wt\cQ_n(z):=\begin{bmatrix}\wt\Theta^{(0)}_n(z)& \wt
C_n(z)\cr \wt B_n(z)&\wt
A_n(z)\end{bmatrix}:\begin{array}{l}\sM\\\oplus\\\sD_{\Gamma^*_{2n+1}}\end{array}\to
\begin{array}{l}\sN\\\oplus\\\sD_{\Gamma_{2n+1}}\end{array}.
\end{equation}

Consider the following discrete time-invariant systems:
\[
\tau_n=\left\{\begin{bmatrix} N_n&M_n\cr L_n&\cS_{n,0}
\end{bmatrix};\;
\sM\oplus \sD_{\Gamma^*_{2n+1}},\; \sN\oplus\sD_{\Gamma_{2n+1}},\;
\cK_n\right\}
\]
and
\[
 \wt\tau_n=\left\{\begin{bmatrix} \wt N_n&\wt M_n\cr \wt
L_n&\wt
\cS_{n,0}\end{bmatrix};\;\sM\oplus\sD_{\Gamma^*_{2n+1}},\;\sN \oplus
\sD_{\Gamma_{2n+1}},\;\wt \cK_n\right\},
\]
where
\[
N_n=\wt N_n=\begin{bmatrix}\Gamma_0&0\cr 0&0 \end{bmatrix}:\begin{array}{l}\sM\\
\oplus\\
\sD_{\Gamma^*_{2n+1}}\end{array}\to \begin{array}{l}\sN\\
\oplus\\
\sD_{\Gamma_{2n+1}}\end{array},
\]
\[
M_n
f=\begin{bmatrix}D_{\Gamma^*_0}\Gamma_1&D_{\Gamma^*_0}D_{\Gamma^*_1}\end{bmatrix}P_{\cH_1}f\oplus
\begin{bmatrix}D_{\Gamma_{2n+1}}&-\Gamma^*_{2n+1}\end{bmatrix}P_{\cH_{n+1}}f\in
\begin{array}{l}\sN\\
\oplus\\
\sD_{\Gamma_{2n+1}}\end{array},\; f\in\cK_n,
\]
\[
L_n\vec \f=D_{\Gamma_0}P_{\sM}\vec\f+P_{\sD_{\Gamma^*_{2n+1}}}\vec\f
=\begin{bmatrix}D_{\Gamma_0}\f_1\cr 0\cr\vdots\cr 0\cr
\f_2\end{bmatrix}\in\cK_{n},\; \vec \f=\begin{bmatrix}\f_1\cr\f_2
\end{bmatrix},\; \f_1\in\sM,\; \f_2\in\sD_{\Gamma^*_{2n+1}},
\]
\[
\wt M_n f=D_{\Gamma^*_0}P_{\sD_{\Gamma^*_0}} f\oplus
P_{\sD_{\Gamma_{2n+1}}}f\in
\begin{array}{l}\sN\\
\oplus\\
\sD_{\Gamma_{2n+1}}\end{array},\; f\in\wt\cK_n,
\]
\[
\wt L_n\vec\f=\begin{bmatrix}{\Gamma_1}\cr
D_{\Gamma_1}\end{bmatrix}D_{\Gamma_0}\f_1\oplus
\begin{bmatrix}D_{\Gamma^*_{2n+1}}\cr-{\Gamma^*_{2n+1}}\end{bmatrix}\f_2\in\wt\cK_{n},\;
\vec \f=\begin{bmatrix}\f_1\cr\f_2
\end{bmatrix},\; \f_1\in\sM,\; \f_2\in\sD_{\Gamma^*_{2n+1}},
\]
Then from \eqref{functt1}, \eqref{cq1}, \eqref{functt11}, and
\eqref{cq2} it follows that $\cQ(z)$ and $\wt\cQ(z)$ are the
transfer functions of the systems $\tau_n$ and $\wt\tau_n,$
respectively.
\begin{proposition}\label{conssys} The discrete time-invariant syaytems $\tau_n$ and $\wt\tau_n$
are conservative and unitary equivalent. Therefore,
\[
\cQ_n=\wt\cQ_n\in\bS\left(\sM\oplus\sD_{\Gamma^*_{2n+1}},\sN\oplus\sD_{\Gamma_{2n+1}}\right).
\]
\end{proposition}
\begin{proof}
The statements follow from definitions of $\tau_n$ and $\wt\tau_n$,
equalities \eqref{prod1}, \eqref{VN}, \eqref{wno}. One can verify
that
\[
\wt M_nV_n=M_n,\; V_n L_n=\wt L_n,\; V_n\cS_{n,0}=\wt \cS_{n,0} V_n,
\]
where $V_n$ is given by \eqref{VN}. This means that $\tau_n$ and
$\wt \tau_n$ are unitary equivalent.
\end{proof}
Proposition \ref{conssys} yields the equalities
\[
\Theta^{(0)}_n(z)=\wt\Theta^{(0)}_n(z),\; A_n(z)=\wt A_n(z),\;
B_n(z)=\wt B_n(z),\; C_n(z)=\wt C_n(z),\; z\in\dD.
\]
Since $B_n(0)=0$, $C_n(0)=0$, and $A_n(0)=0$, we get
\[
||B_n(z)||\le|z|, \; ||C_n(z)||\le|z|,\;||A_n(z)||\le|z|,\; z\in\dD.
\]
\begin{theorem}
\label{desr} Let $\Theta\in{\bS}(\sM,\sN)$ and let
$\{\Gamma_k\}_{k\ge 0}$ be its Schur parameters. Suppose that
$\sD_{\Gamma_{2n+1}}\ne \{0\}$ and $\sD_{\Gamma^*_{2n+1}}\ne \{0\}$
for some $n$. Then the functions $\Theta$ and $\Theta_{2n+2}$ are
connected by the relations
\[
\begin{array}{l}
\Theta(z)=\Theta_n^{(0)}(z)+C_n(z)\left(I_{\sD_{\Gamma^*_{2n+1}}}-\Theta_{2n+2}(z)A_n(z)\right)^{-1}
\Theta_{2n+2}(z)B_n(z)\\
=\Theta_n^{(0)}(z)+C_n(z)\Theta_{2n+2}(z)\left(I_{\sD_{\Gamma_{2n+1}}}-A_n(z)\Theta_{2n+2}(z)
\right)^{-1} B_n(z),
\end{array}
\]
where the entries of the Schur class function
$\cQ_n(z)=\begin{bmatrix}\Theta^{(0)}_n(z)& C_n(z)\cr
B_n(z)&A_n(z)\end{bmatrix}$ are given by \eqref{functt1}.
\end{theorem}
\begin{proof}
Use \eqref{charfunc1}, \eqref{charfunc2}, and apply \eqref{schfr1},
\eqref{schfr2}, \eqref{expr11}, \eqref{expr22}, and Proposition
\ref{conssys}.
\end{proof}
We arrive at the following statement.
\begin{theorem}
\label{schpr1}
 Let
$$\Gamma_0\in\bL(\sM,\sN),\;
\Gamma_1,\ldots,\; \Gamma_{2n+1}$$
 be a choice sequence. Suppose $
\sD_{\Gamma_{2n+1}}\ne\{0\},\;\sD_{\Gamma^*_{2n+1}}\ne\{0\}.
$
 Then the formula
\begin{equation}
\label{solut}
\Theta(z)=\Theta_n^{(0)}(z)+C_n(z)\cE(z)\left(I_{\sD_{\Gamma_{2n+1}}}-A_n(z)\cE(z)\right)^{-1}
B_n(z),\; z\in\dD
\end{equation}
gives a one-to-one correspondence between all functions
$\cE\in\bS(\sD_{\Gamma_{2n+1}},\sD_{\Gamma^*_{2n+1}})$ and all
functions $\Theta\in\bS(\sM,\sN)$ having given choice sequence
$\Gamma_0,\Gamma_1,\ldots,\Gamma_{2n+1}$ as their first $2n+2$ Schur
parameters. Moreover the Schur parameters of the function $\Theta$
given by \eqref{solut} are
\[
\Gamma_0,\Gamma_1,\ldots,\Gamma_{2n+1},\gamma^{(\cE)}_0,\gamma^{(\cE)}_1,\ldots,
\]
where
$\gamma^{(\cE)}_0\in\bL(\sD_{\Gamma_{2n+1}},\sD_{\Gamma^*_{2n+1}}),\gamma^{(\cE)}_1,\ldots$
are the Schur parameters of $\cE(z)$.
\end{theorem}
\begin{proof}
Assume $\Theta\in\bS(\sM,\sN)$ has
$\Gamma_0,\Gamma_1,\ldots,\Gamma_{2n+1}$ as its first $2n+2$ Schur
parameters and let $\Gamma_{2n+2},\ldots$ are the rest Schur
parameters of $\Theta$. Denote $\cE$ the function from
$\bS(\sD_{\Gamma_{2n+1}},\sD_{\Gamma^*_{2n+1}})$ with the Schur
parameters $\Gamma_{2n+2},\ldots$. Then $\Theta_{2n+2}(z)=\cE(z)$
for all $z\in\dD$. Here $\Theta_{2n+2}$ is the function associated
with $\Theta$ in accordance with the Schur algorithm. Then
constructing the block operator matrix $\cS_{n,0}$ by means  of
\eqref{prdtn}, the function $\cQ_n(z)$ of the form \eqref{cq1}, and
applying Theorem \ref{desr}, we get equality \eqref{solut}.

Conversely, suppose
$\cE\in\bS(\sD_{\Gamma_{2n+1}},\sD_{\Gamma^*_{2n+1}})$ is given. Let
$$\gamma^{(\cE)}_0\in\bL(\sD_{\Gamma_{2n+1}},\sD_{\Gamma^*_{2n+1}}),\;\gamma^{(\cE)}_1
,\ldots
$$
be the Schur parameters of $\cE$. Let $\Theta\in\bS(\sM,\sN)$ be the
function with the Schur parameters
\[
\Gamma_0,\ldots, \Gamma_{2n+1},\gamma^{(\cE)}_0,\ldots.
\]
Then $\Theta_{2n+2}(z)=\cE(z)$ for all $z\in\dD$ and by Theorem
\ref{desr} the functions $\cE(z)$ and $\Theta(z)$ are connected by
\eqref{solut}.
\end{proof}
Observe that the function $\Theta^{(0)}_n\in\bS(\sM,\sN)$
corresponds to the parameter
$$\cE\equiv 0\in\bS(\sD_{\Gamma_{2n+1}},\sD_{\Gamma^*_{2n+1}}),$$ i.e., the
Shur parameters of $\Theta^{(0)}_n$ are
$\Gamma_0,\ldots,\Gamma_{2n+1},0,0,\ldots.$
\begin{corollary}
\label{unq1} Let
$$\Gamma_0\in\bL(\sM,\sN),\;
\Gamma_1,\ldots,\; \Gamma_{2n+1},\Gamma_{2n+2}$$
 be a choice sequence. Suppose $\Gamma_{2n+2}$ is either isometry or
co-isometry.
 Then
\begin{equation}
\label{solutun}
\Theta(z)=\Theta_n^{(0)}(z)+C_n(z)\Gamma_{2n+2}\left(I_{\sD_{\Gamma_{2n+1}}}-A_n(z)\Gamma_{2n+2}\right)^{-1}
B_n(z),\; z\in\dD
\end{equation}
is a unique function from $\bS(\sM,\sN)$ having
$\Gamma_0,\Gamma_1,\ldots\Gamma_{2n+2}$ as its first $2n+2$ Schur
parameters.
\end{corollary}

\subsubsection{CMV block operator matrices and the coupling of
conservative systems} Let $\{\Gamma_k\}_{k\ge} 0$ be a choice
sequence, $\Gamma_0\in\bL(\sM,\sN)$. Suppose condition
\eqref{nevir}. Let $\cU_0=\cU_0(\{\Gamma_k\})$ be CMV block operator
matrix. The unitary operator $U_0$ acts from the Hilbert space
$\sM\oplus\sH_0$ onto the Hilbert space $\sN\oplus\sH$, where
$\sH_0$ is a Hilbert space constructed by means of defect spaces
$\{\sD_{\Gamma_k},\;\sD_{\Gamma^*_k}\}_{k\ge 0}$ (see Section
\ref{cmv} and Appendix). The operator $\cU_0$ takes the form
\[
\cU_0=\left[\begin{array}{c|c}\Gamma_0&\begin{array}{cccc}D_{\Gamma_0}\Gamma_1&D_{\Gamma_0}D_{\Gamma^*_1}&0&\ldots
\end{array}\\
\hline\begin{array}{c}D_{\Gamma_0}\cr
0\cr\vdots\end{array}&\cT_0\end{array}\right],
\]
where $\cT_0$ is truncted CMV matrix. Let the Hilbert space $\cK_n$
be defined in \eqref{HSN}. Consider the conservative system $\tau_n$
defined in Subsection \ref{conn11},
\[
\tau_n=\left\{\Psi_n;\; \sM\oplus \sD_{\Gamma^*_{2n+1}},\;
\sN\oplus\sD_{\Gamma_{2n+1}},\; \cK_n\right\}.
\]
The corresponding unitary operator $\Psi_n=\begin{bmatrix}
N_n&M_n\cr L_n&\cS_{n,0}
\end{bmatrix}$
 is of the form
\[
\begin{array}{l} \left[\begin{array}{c|c}
\begin{array}{cccccc}\Gamma_0&0\cr 0&0
\end{array}&
\begin{array}{ccccccc}
D_{\Gamma^*_0}\Gamma_1&D_{\Gamma^*_0}D_{\Gamma^*_1}
&0&\ldots&0&0&0\cr
0&0&\ldots&\ldots&0&D_{\Gamma_{2n+1}}&-\Gamma^*_{2n+1}\end{array}\\
 \hline
\begin{array}{cc}
D_{\Gamma_{0}}&0 \cr 0&0\cr
 \vdots&\vdots\cr
0& I_{\sD_{\Gamma^*_{2n+1}}}
\end{array}&
\cS_{n,0} 
\end{array}
\right]:\\
\qquad\qquad\qquad\qquad\qquad\qquad\begin{array}{c}\left[\begin{array}{l}\sM\\\oplus\\\sD_{\Gamma^*_{2n+1}}\end{array}\right]\\\oplus\\\cK_n\end{array}
\to
\begin{array}{c}\left[\begin{array}{l}\sM\\\oplus\\\sD_{\Gamma^*_{2n+1}}\end{array}\right]\\\oplus\\\cK_n\end{array}.
\end{array}
\]
Consider also the CMV matrix $\cU_{2n+2}\left(\{\Gamma_k\}_{k\ge
2n+2}\right)$ . The precise form of $\cU_{2n+2}$ depends on the
cases 1)-- 7) mentioned in the proof of Theorem \ref{formula1}. In
particular, if both operators $\Gamma_{2n+2}$ and $\Gamma_{2n+3}$
are neither isometric nor co-isometric, then
\begin{multline*}
\cU_{2n+2}=\\
=\left[\begin{array}{c|c}\Gamma_{2n+2}&\begin{array}{cccc}D_{\Gamma^*_{2n+2}}\Gamma_{2n+3}&
D_{\Gamma^*_{2n+2}}\sD_{\Gamma^*_{2n+3}}&0&\ldots\end{array}\\
\hline\begin{array}{c}D_{\Gamma_{2n+2}}\cr
0\cr\vdots\end{array}&\cT_{2n+2}
\end{array}
\right]: 
\begin{array}{l}\sD_{\Gamma_{2n+1}}\\\oplus\\\sH_{2n+2}\end{array}\to
\begin{array}{l}\sD_{\Gamma^*_{2n+1}}\\\oplus\\\sH_{2n+2}\end{array},
\end{multline*}
where $\cT_{2n+2}=\cT'$ is the truncated CMV matrix related to
$\cU_{2n+2}$. Let
$$\zeta_{2n+2}=\left\{\cU_{2n+2};\;\sD_{\Gamma_{2n+1}},\sD_{\Gamma^*_{2n+1}},\;\sH_{2n+2}\right\}$$
be the corresponding conservative system,
$\sH_{2n+2}=\sH_{2n+2}\left(\{\Gamma_k\}_{k\ge
2n+2}\right)(=\sH_0')$ (see \eqref{matrrep}),
$\sH_0=\cK_n\oplus\sH_{2n+2}.$ The truncated CMV block operator
matrix $\cT_0$ with respect to the decomposition
$\sH_0=\cK_n\oplus\sH_{2n+2}$ takes the form \eqref{matrrep}:
\[
\cT_0= \left[\begin{array}{c|c} \cS_n&
\begin{array}{ccc}0&0&\ldots
\cr \vdots&\vdots&\vdots\cr 0&0&\ldots\cr
\cC_{n+1}&0&\ldots\end{array}\\
\hline\begin{array}{cccc}0&\ldots&0&\cA_{n+1} \cr
0&\ldots&0&0\cr\vdots&\vdots&\vdots&\vdots\end{array}&\cT_{2n+2}
\end{array}\right]
\]
 The notion of the coupling of unitary
colligations (conservative systems) can be found in 
\cite{Kh3}, \cite{Kh2}, \cite{Peller}. By  straightforward
calculations one can verify that the conservative system (unitary
colligation) $\tau_n$ is the \textit{universal} and \textit{the
conservative system $\zeta_0$ (see \eqref{CMVMODEL}) is the coupling
of the conservative systems $\tau_n $ and $\zeta_{2n+2}$}. This
means that the equality
\[
\begin{bmatrix}\mathfrak n\cr h'_0\cr h'_1\end{bmatrix}=U_0\begin{bmatrix}\mathfrak m\cr h_0\cr
h_1\end{bmatrix},\;\mathfrak m\in\sM,\;\mathfrak n\in\sN, \; h_0,
h'_0\in\cK_n,\; h_1,\; h_1'\in\sH_{2n+2}
\]
holds if
\[
\Psi_n\begin{bmatrix}\mathfrak m\cr \gamma_*\cr
h_0\end{bmatrix}=\begin{bmatrix}\mathfrak n\cr \gamma\cr
h'_0\end{bmatrix} \quad\mbox{and}\quad
\cU_{2n+2}\begin{bmatrix}\gamma\cr
h_1\end{bmatrix}=\begin{bmatrix}\gamma_*\cr h'_1\end{bmatrix}
\]
for some $\gamma_*\in\sD_{\Gamma^*_{2n+1}}$ and
$\gamma\in\sD_{\Gamma_{2n+1}}$. Notice that
\begin{multline*}
\gamma=\left(D_{\Gamma_{2n+1}}P_{\sD_{\Gamma_{2n}}}-\Gamma^*_{2n+1}P_{\sD_{\Gamma^*_{2n+1}}}\right)h_0,\\
\gamma_*=\Gamma_{2n+2}\left(D_{\Gamma_{2n+1}}P_{\sD_{\Gamma_{2n}}}-\Gamma^*_{2n+1}P_{\sD_{\Gamma^*_{2n+1}}}\right)h_0
\\
+D_{\Gamma^*_{2n+2}}\left(\Gamma_{2n+3}P_{\sD_{\Gamma_{2n+3}}}+D_{\Gamma^*_{2n+3}}P_{\sD_{\Gamma^*_{2n+3}}}
\right)h_1.
\end{multline*}
As it is established in \cite{ArGr}, \cite{Kh3}, \cite{Kh},
\cite{Kh2} if a conservative system $\Sigma$ is the coupling of
certain universal conservative system $\Sigma_0$ and a conservative
system $\Sigma'$, then the transfer functions $\Theta_\Sigma$ and
$\Theta_{\Sigma'}$ of the systems $\Sigma$ and $\Sigma'$,
respectively, are connected by the relation
\[
\Theta_\Sigma(z)=a_{11}(z)+a_{12}(z)\left(I-\Theta_{\Sigma'}(z)a_{22}(z)\right)^{-1}
\Theta_{\Sigma'}(z)a_{21}(z),\; z\in\dD,
\]
where
\[
\Theta_{\Sigma_0}(z)=\begin{bmatrix}a_{11}(z)& a_{12}(z)\cr
a_{21}(z)& a_{22}(z)\end{bmatrix}
\]
 is the transfer function of the
universal system $\Sigma_0$. Thus, relations in Theorem \ref{desr}
are also a consequence of the facts that $\zeta_0$ is the coupling
of $\tau_n$ and $\zeta_{2n+2}$ and the unitary equivalence of the
systems $\tau_n$ and $\wt \tau_n$ (see Proposition \ref{conssys}).
\subsection{The matrix $\wh \cS_{n,0}$ and connection between $\Theta$ and $\Theta_{2n+1}$}
Now we established a connection between $\Theta$ and
$\Theta_{2n+1}$. Suppose $\{\Gamma_k\}$ are the Schur parameters of
$\Theta\in\bS(\sM,\sN)$ and
$\sD_{\Gamma_{2n}}\ne\{0\},\;\sD_{\Gamma^*_{2n}} \ne\{0\}.$ Then the
operators (the choice sequence)
\[
\wh\Gamma_0:=0\in\bL(\sM,\sN),\; \wh
\Gamma_1:=\Gamma_0\in\bL(\sM,\sN),\;\wh\Gamma_2:=\Gamma_1\in\bL(\sD_{\Gamma_0},\sD_{\Gamma^*_0}),\;
\ldots
\]
are the Schur parameters of the function
$\wh\Theta(z)=z\Theta(z)\in\bS(\sM,\sN)$. So,
$\wh\Gamma_l:=\Gamma_{l-1}$, $l\ge 1$. Let
\[
\wh
\cH_1=\begin{array}{l}\sM\\\oplus\\\sD_{\Gamma^*_{0}}\end{array},\;
\wh\cH_k:=\begin{array}{l}\sD_{\wh\Gamma_{2k-2}}\\\oplus\\\sD_{\wh\Gamma^*_{2k-1}}\end{array}=
\begin{array}{l}\sD_{\Gamma_{2k-3}}\\\oplus\\\sD_{\Gamma^*_{2k-2}}\end{array},\;
k=2,\ldots, n+1,
\]
\[
\wh\cK_n:=\bigoplus\limits_{k=1}^{n+1}\wh\cH_{k}.
\]
Now we can apply the approach of Subsection \ref{conn11}.
 Define the
operator $\wh \cS_{n,0}$ in accordance with \eqref{tgeo}
\[
\wh \cS_{n,0}:=
\cS_n\left(\wh\Gamma_0,\wh\Gamma_1,\dots,\wh\Gamma_{2n+1}, 0\right)=
\cS_{n}\left(0, \Gamma_0, \Gamma_1,\ldots, \Gamma_{2n},0 \right),
\]
i.e.,
\begin{multline*}
\wh
\cS_{n,0}=\left(-\wh\Gamma^*_0\oplus\bigoplus\limits_{k=1}^{n}{\bf
J}_{\wh\Gamma_{2k}}\oplus 0\right)\times \left(
\bigoplus\limits_{k=1}^{n+1}{\bf J}_{\wh\Gamma_{2k-1}}\right)\\
= \left(0\oplus\bigoplus\limits_{k=1}^{n}{\bf
J}_{\Gamma_{2k-1}}\oplus 0\right)\times \left(
\bigoplus\limits_{k=1}^{n+1}{\bf J}_{\Gamma_{2k-2}}\right).
\end{multline*}
Then construct the function
$$\wh \cQ_n(z)=\begin{bmatrix}\wh\Theta^{(0)}_n(z)& \wh C_n(z)\cr
\wh B_n(z)&\wh
A_n(z)\end{bmatrix}:\begin{array}{l}\sM\\\oplus\\\sD_{\Gamma^*_{2n}}\end{array}\to
\begin{array}{l}\sN\\\oplus\\\sD_{\Gamma_{2n}}\end{array}$$
in accordance with \eqref{functt1} and \eqref{cq1}. We have
\begin{equation}
\label{functt222} \left\{
\begin{array}{l}
\wh\Theta^{(0)}_n(z):=z\begin{bmatrix}\Gamma_0&D_{\Gamma^*_0}\end{bmatrix}
\left(P_{\wh\cH_1}\left(I_{\wh\cK_n}-z\wh \cS_{n,0}\right)^{-1}\uphar\sM\right)\\\
\qquad\qquad\qquad\qquad\in\bL(\sM,\sN),\\
\wh A_n(z):=z\begin{bmatrix}D_{\Gamma_{2n}}&-\Gamma^*_{2n}
\end{bmatrix}\left(P_{\wh\cH_{n+1}}\left(I_{\wh\cK_n}-z\wh
\cS_{n,0}\right)^{-1}\uphar\sD_{\Gamma^*_{2n}}\right)\\
\qquad\qquad\qquad\qquad\in\bL(\sD_{\Gamma^*_{2n}},\sD_{\Gamma_{2n}}), \\
\wh B_n(z):=z\begin{bmatrix}D_{\Gamma_{2n}}&-\Gamma^*_{2n}
\end{bmatrix}\left(P_{\wh\cH_{n+1}}\left(I_{\wh\cK_n}-z\wh
\cS_{n,0}\right)^{-1}\uphar\sM\right)\\
\qquad\qquad\qquad\qquad\in\bL(\sM,\sD_{\Gamma_{2n}}),\\
\wh C_n(z):=z\begin{bmatrix}\Gamma_0&D_{\Gamma^*_0}\end{bmatrix}
\left(P_{\wh\cH_1}\left(I_{\wh\cK_n}-z\wh
\cS_{n,0}\right)^{-1}\uphar\sD_{\Gamma^*_{2n}}\right)\\
\qquad\qquad\qquad\qquad \in\bL(\sD_{\Gamma^*_{2n}},\sN).
\end{array}\right.
\end{equation}
Due to Proposition \ref{conssys} the function $\wh\cQ_n$ belongs to
the Schur class
$\bS\left(\sM\oplus\sD_{\Gamma^*_{2n}},\sN\oplus\sD_{\Gamma_{2n}}\right).$
 Since $||\wh \cQ_n(z)||\le 1$ for all $z\in\dD$ and $\wh
\cQ_n(0)=0$, by Schwarz's lemma for the function
\[
\wh q_n(z):=z^{-1}\wh \cQ_n(z)=\begin{bmatrix}z^{-1}\wh
\Theta^{(0)}_n(z)& z^{-1}\wh C_n(z)\cr z^{-1}\wh B_n(z)&z^{-1}\wh
A_n(z)\end{bmatrix}:\begin{array}{l}\sM\\\oplus\\\sD_{\Gamma^*_{2n}}\end{array}\to
\begin{array}{l}\sN\\\oplus\\\sD_{\Gamma_{2n}}\end{array}
\]
 we obtain $ ||\wh q_n(z)||\le 1,\;z\in\dD.$
Clearly, the function
\begin{equation}
\label{newq} q_n(z):=\begin{bmatrix}z^{-1}\wh \Theta^{(0)}_n(z)&
z^{-1}\wh C_n(z)\cr \wh B_n(z)&\wh
A_n(z)\end{bmatrix}:\begin{array}{l}\sM\\\oplus\\\sD_{\Gamma^*_{2n}}\end{array}\to
\begin{array}{l}\sN\\\oplus\\\sD_{\Gamma_{2n}}\end{array}
\end{equation}
is also from the Schur class. Set
\[
\theta^{(0)}_n:=z^{-1}\wh \Theta^{(0)}_n(z),\; c_n(z):=z^{-1}\wh
C_n(z),\; a_n(z):=\wh A_n(z), \; b_n(z):=\wh B_n(z),\; z\in\dD.
\]
So,
\begin{equation}
\label{functt2} \left\{
\begin{array}{l}
\theta^{(0)}_n(z):=\begin{bmatrix}\Gamma_0&D_{\Gamma^*_0}\end{bmatrix}
\left(P_{\wh\cH_1}\left(I_{\wh\cK_n}-z\wh \cS_{n,0}\right)^{-1}\uphar\sM\right)\\
\qquad\qquad\qquad\qquad\in\bL(\sM,\sN),\\
a_n(z):=z\begin{bmatrix}D_{\Gamma_{2n}}&-\Gamma^*_{2n}
\end{bmatrix}\left(P_{\wh\cH_{n+1}}\left(I_{\wh\cK_n}-z\wh
\cS_{n,0}\right)^{-1}\uphar\sD_{\Gamma^*_{2n}}\right)\\
\qquad\qquad\qquad\qquad\in\bL(\sD_{\Gamma^*_{2n}},\sD_{\Gamma_{2n}}), \\
b_n(z):=z\begin{bmatrix}D_{\Gamma_{2n}}&-\Gamma^*_{2n}
\end{bmatrix}\left(P_{\wh\cH_{n+1}}\left(I_{\wh\cK_n}-z\wh
\cS_{n,0}\right)^{-1}\uphar\sM\right)\\
\qquad\qquad\qquad\qquad
\in\bL(\sM,\sD_{\Gamma_{2n}}),\\
c_n(z):=\begin{bmatrix}\Gamma_0&D_{\Gamma^*_0}\end{bmatrix}
\left(P_{\wh\cH_1}\left(I_{\wh\cK_n}-z\wh
\cS_{n,0}\right)^{-1}\uphar\sD_{\Gamma^*_{2n}}\right)\\
\qquad\qquad\qquad\qquad \in\bL(\sD_{\Gamma^*_{2n}},\sN).
\end{array}\right.
\end{equation}
\begin{theorem}
\label{connect21} Let $\Theta\in{\bS}(\sM,\sN)$ and let
$\{\Gamma\}_{n\ge 0}$ be its Schur parameters. Suppose that
$\sD_{\Gamma_{2n}}\ne \{0\}$ and $\sD_{\Gamma^*_{2n}}\ne \{0\}$.
Then the functions $\Theta$ and $\Theta_{2n+1}$ are connected by the
relations
\begin{equation}
\label{connectodd}
\begin{array}{l}
\Theta(z)=\theta_n^{(0)}(z)+c_n(z)\left(I_{\sD_{\Gamma^*_{2n}}}-\Theta_{2n+1}(z)a_n(z)\right)^{-1}
\Theta_{2n+1}(z)b_n(z)\\
=\theta_n^{(0)}(z)+c_n(z)\Theta_{2n+1}(z)\left(I_{\sD_{\Gamma_{2n}}}-a_n(z)\Theta_{2n+1}(z)
\right)^{-1} b_n(z), \;z\in\dD,
\end{array}
\end{equation}
where the entries of the Schur class function
$$q_n(z)=\begin{bmatrix}\theta^{(0)}_n(z)&
c_n(z)\cr
b_n(z)&a_n(z)\end{bmatrix}:\begin{array}{l}\sM\\\oplus\\\sD_{\Gamma^*_{2n}}\end{array}\to
\begin{array}{l}\sN\\\oplus\\\sD_{\Gamma_{2n}}\end{array}$$
are given by \eqref{functt2}.
\end{theorem}
\begin{proof} Taking into account that
$\Theta_{2n+1}(z)=\wh\Theta_{2n+2}(z)$, $z\in\dD$ and applying
Theorem \ref{desr} we obtain
\[
\begin{array}{l}
\wh \Theta(z)=z\Theta(z)=\wh\Theta_n^{(0)}(z)+\wh
C_n(z)\left(I_{\sD_{\Gamma^*_{2n}}}-z\wh\Theta_{2n+2}(z)\wh
A_n(z)\right)^{-1}
\wh\Theta_{2n+2}(z)\wh B_n(z)\\
=\wh\Theta_n^{(0)}(z)+\wh C_n(z)\wh
\Theta_{2n+2}(z)\left(I_{\sD_{\Gamma_{2n}}}-z\wh
A_n(z)\wh\Theta_{2n+2}(z) \right)^{-1} \wh B_n(z).
\end{array}
\]
Then from \eqref{functt222}, \eqref{newq}, and \eqref{functt2}
 we get \eqref{connectodd}.
\end{proof}
\begin{theorem}
\label{schpr2} Let
$$\Gamma_0\in\bL(\sM,\sN),\;
\Gamma_1,\ldots,\; \Gamma_{2n}$$
 be a choice sequence. Suppose
$
\sD_{\Gamma_{2n}}\ne\{0\},\;\sD_{\Gamma^*_{2n}}\ne\{0\}.
$
 Then the formula
\begin{equation}
\label{solut1}
\Theta(z)=\theta_n^{(0)}(z)+c_n(z)\cE(z)\left(I_{\sD_{\Gamma_{2n}}}-a_n(z)\cE(z)\right)^{-1}
b_n(z),\; z\in\dD
\end{equation}
gives a one-to-one correspondence between all functions
$\cE\in\bS(\sD_{\Gamma_{2n}},\sD_{\Gamma^*_{2n}})$ and all functions
$\Theta\in\bS(\sM,\sN)$ having given choice sequence
$\Gamma_0,\Gamma_1,\ldots,\Gamma_{2n}$ as their first $2n+1$ Schur
parameters. Moreover the Schur parameters of the function $\Theta$
given by \eqref{solut1} are
\[
\Gamma_0,\Gamma_1,\ldots,\Gamma_{2n},\gamma^{(\cE)}_0,\gamma^{(\cE)}_1,\ldots,
\]
where
$\gamma^{(\cE)}_0\in\bL(\sD_{\Gamma_{2n}},\sD_{\Gamma^*_{2n}}),\gamma^{(\cE)}_1,\ldots$
are the Schur parameters of $\cE(z)$.
\end{theorem}
\begin{corollary}
\label{unq2} Let
$$\Gamma_0\in\bL(\sM,\sN),\;
\Gamma_1,\ldots,\; \Gamma_{2n+1}$$
 be a choice sequence. Suppose $\Gamma_{2n+1}$ is either isometry or
co-isometry.
 Then
\begin{equation}
\label{solutun2}
\Theta(z)=\theta_n^{(0)}(z)+c_n(z)\Gamma_{2n+1}\left(I_{\sD_{\Gamma_{2n}}}-a_n(z)\Gamma_{2n+1}\right)^{-1}
b_n(z),\; z\in\dD
\end{equation}
is a unique function from $\bS(\sM,\sN)$ having
$\Gamma_0,\Gamma_1,\ldots\Gamma_{2n+1}$ as its first $2n+1$ Schur
parameters.

\end{corollary}

\section{Descriptions of all solutions to the Schur problem} \label{reshen}

We describe the algorithm for the solutions to the Schur problem
involving sub-matrices of block-operator CMV matrices.

Let the Schur sequence $C_0,\dots, C_N\in \bL(\sM,\sN)$ be given.
Calculate $(D^2_{T_N})_\sM$ and $(D^2_{\wt T_N})_\sN$, where $T_N$
and $\wt T_N$ are the Toeplitz matrices of the form \eqref{toepn}
and \eqref{wttoepn}.

Suppose $(D^2_{T_N})_\sM\ne 0$ and $(D^2_{\wt T_N})_\sN\ne 0$. Find
the choice sequence
$$\Gamma_0=C_0, \;\Gamma_1, \ldots,\Gamma_N,$$
corresponding to the data $\{C_k\}_{k=0}^N$ (see Subsection
\ref{shseq}). Then any two solutions of the Schur problem differ by
the Schur parameters, which start with the number $N+1$.

If $N=2n+1$, find the matrix $\cS_{n,0}$ (see \eqref{tgeo}) and
calculate the functions \eqref{functt1}. The formula \eqref{solut}
gives all solutions to the Schur problem.

If $N=2n$, then calculate the matrix $\wh \cS_{n,0}$ constructed by
means of the choice sequence
\[
0\in\bL(\sM,\sN),\;\Gamma_0,\ldots, \Gamma_N,
\]
and calculate the functions \eqref{functt2}. The formula
\eqref{solut1} gives all solutions.

Thus, all solutions are given by the fractional linear
transformation
\begin{equation}
\label{opresh}
\begin{array}{l}
\Theta(z)=\Theta^{(0)}_N(z)+C_N(z)\cE(z)\left(I_{\sD_{\Gamma_N}}-A_N(z)\cE(z)\right)^{-1}B_N(z)\\
\qquad=\Theta^{(0)}_N(z)+C_N(z)\left(I_{\sD_{\Gamma^*_N}}-\cE(z)A_N(z)\right)^{-1}\cE(z)B_N(z),
\end{array}
\end{equation}
where $\cE(z)$ is an arbitrary function from
$\bS(\sD_{\Gamma_N},\sD_{\Gamma^*_N})$,

\[
\Theta^{(0)}_N=\left\{\begin{array}{l}\Gamma_0+zD_{\Gamma^*_0}\begin{bmatrix}\Gamma_1&D_{\Gamma^*_1}\end{bmatrix}
\left(P_{\cH_1}\left(I_{\cK_n}-z\cS_{n,0}\right)^{-1}\uphar\sD_{\Gamma_0}\right)D_{\Gamma_0},\;
N=2n+1\\
\begin{bmatrix}\Gamma_0&D_{\Gamma^*_0}\end{bmatrix}
\left(P_{\wh\cH_1}\left(I_{\wh\cK_n}-z\wh
\cS_{n,0}\right)^{-1}\uphar\sM\right),\; N=2n
\end{array}\right.,
\]
\[
C_N(z)=\left\{\begin{array}{l}zD_{\Gamma^*_0}\begin{bmatrix}\Gamma_1&D_{\Gamma^*_1}\end{bmatrix}
\left(P_{\cH_1}\left(I_{\cK_n}-z\cS_{n,0}\right)^{-1}\uphar\sD_{\Gamma^*_{2n+1}}\right),\;
N=2n+1\\
\begin{bmatrix}\Gamma_0&D_{\Gamma^*_0}\end{bmatrix}
\left(P_{\wh\cH_1}\left(I_{\wh\cK_n}-z\wh
\cS_{n,0}\right)^{-1}\uphar\sD_{\Gamma^*_{2n}}\right),\; N=2n
\end{array}\right.,
\]
\[
A_N(z)=\left\{\begin{array}{l}z\begin{bmatrix}D_{\Gamma_{2n+1}}&-\Gamma^*_{2n+1}
\end{bmatrix}\left(P_{\cH_{n+1}}\left(I_{\cK_n}-z\cS_{n,0}\right)^{-1}\uphar\sD_{\Gamma^*_{2n+1}}\right),\;
N=2n+1\\
z\begin{bmatrix}D_{\Gamma_{2n}}&-\Gamma^*_{2n}
\end{bmatrix}\left(P_{\wh\cH_{n+1}}\left(I_{\wh\cK_n}-z\wh
\cS_{n,0}\right)^{-1}\uphar\sD_{\Gamma^*_{2n}}\right),\; N=2n
\end{array}\right.,
\]
\[
B_N(z)=\left\{\begin{array}{l}z\begin{bmatrix}D_{\Gamma_{2n+1}}&-\Gamma^*_{2n+1}
\end{bmatrix}\left(P_{\cH_{n+1}}\left(I_{\cK_n}-z\cS_{n,0}\right)^{-1}\uphar\sD_{\Gamma_{0}}\right)D_{\Gamma_0},\;
N=2n+1\\
z\begin{bmatrix}D_{\Gamma_{2n}}&-\Gamma^*_{2n}
\end{bmatrix}\left(P_{\wh\cH_{n+1}}\left(I_{\wh\cK_n}-z\wh
\cS_{n,0}\right)^{-1}\uphar\sM\right),\; N=2n
\end{array}\right.,
\]
and the operator-valued function for $z\in\dD$
\[
Q_N(z)=\begin{bmatrix}\Theta^{(0)}_N(z)& C_N(z)\cr
B_N(z)&A_N(z)\end{bmatrix}:\begin{array}{l}\sM\\\oplus\\\sD_{\Gamma^*_{N}}\end{array}\to
\begin{array}{l}\sN\\\oplus\\\sD_{\Gamma_{N}}\end{array}
\]
belongs to the Schur class.
Observe that the function $\Theta^{(0)}_N,$ corresponding to the
parameter $\cE\equiv 0\in\bL(\sD_{\Gamma_{N}},\sD_{\Gamma^*_{N}})$ in \eqref{opresh},
is the central solutions to the Schur problems \cite{DFK},
\cite{FrKirLas}, \cite{ArlIEOT2011}.

Parametrization \eqref{opresh} is similar to known parameterizations
\cite{Arov1995}, \cite{BC}, \cite{Dym}, \cite{FoFr},
\cite{FoFrGoKa}, \cite{FrKirLas}, \cite{KKhYu1987} which are
obtained by another methods.

 Suppose $(D^2_{T_N})_\sM= 0$. Find $p$,
$p\le N$ such that $(D^2_{T_p})_\sM= 0,$ but $(D^2_{T_{p-1}})_\sM\ne
0$. Then using \eqref{solutun} for $p=2n+2$ or \eqref{solutun2} for
$p=2n+1$ we get
\[
\begin{array}{l}
\Theta(z)=\Theta^{(0)}_{p-1}(z)+C_{p-1}(z)\Gamma_p\left(I_{\sD_{\Gamma_{p-1}}}-A_{p-1}(z)\Gamma_p\right)^{-1}B_{p-1}
(z)\\
\qquad=\Theta^{(0)}_{p-1}(z)+C_{p-1}(z)\left(I_{\sD_{\Gamma^*_{p-1}}}-\Gamma_p
A_{p-1}(z)\right)^{-1}\Gamma_pB_{p-1}(z).
\end{array}
\]
The case $(D^2_{\wt T_N})_\sN=0$ is similar to the previous one.

\appendix
\section
{Special cases of block operator CMV matrices } \label{REST}

 Let $\{\Gamma_n\}$ be the Schur parameters of the function
$\Theta\in {\bf S}(\sM,\sN)$. Suppose  $\Gamma_m$ is an isometry
(respect., co-isometry, unitary) for some $m\ge 0$. Then
$\Theta_m(z)=\Gamma_m$ for all $z\in\dD$ and
\[
\begin{array}{l}
\Theta_{m-1}(z)=\Gamma_{m-1}+z
D_{\Gamma^*_{m-1}}\Gamma_m(I_{\sD_{\Gamma_{m-1}}}+z\Gamma^*_{m-1}\Gamma_m)^{-1}D_{\Gamma_{m-1}},\\
\Theta_{m-2}(z)=\Gamma_{m-2}+z
D_{\Gamma^*_{m-2}}\Theta_{m-1}(z)(I_{\sD_{\Gamma_{m-2}}}+z\Gamma^*_{m-2}\Theta_{m-1}(z))^{-1}D_{\Gamma_{m-2}},\\
\ldots \quad\ldots\quad \ldots\quad \ldots\quad \ldots \quad
\ldots\quad \ldots\quad \ldots\quad \ldots\quad \ldots\quad
\ldots\quad \ldots,\\
\Theta(z)=\Gamma_{0}+z
D_{\Gamma^*_{0}}\Theta_{1}(z)(I_{\sD_{\Gamma_0}}+z\Gamma^*_{0}\Theta_{1}(z))^{-1}D_{\Gamma_{0}},\;
z\in\dD.
\end{array}
\]
The function $\Theta$ is the transfer function of the simple
conservative systems constructed by means of its Schur parameters
$\{\Gamma_n\}$ and the corresponding block operator CMV matrices
$\cU_0$ and $\wt\cU_0$ \cite{ArlMFAT2009}. Here we present the
explicit form of block operator CMV and truncated CMV matrices. In
particular we revise some misprints in \cite{ArlMFAT2009}.  Notice
that if $\Gamma_m$ is isometric (respect., co-isometric), then
\begin{enumerate}
\item
$\sD_{\Gamma^*_n}=\sD_{\Gamma^*_{m}}$,
$D_{\Gamma^*_{n}}=I_{\sD_{\Gamma^*_{m}}},$
$\Gamma_n=0:\{0\}\to\sD_{\Gamma^*_{m}}$ for $n> m$ (respect.,
$\sD_{\Gamma_n}=\sD_{\Gamma_{m}}$,
$D_{\Gamma_{n}}=I_{\sD_{\Gamma_{m}}}$, $\Gamma_n=0:\sD_{\Gamma_n}\to
\{0\}$ for $n> m$);
\item
 in the definitions of the state spaces
$\sH_0=\sH_0(\{\Gamma_n\}_{n\ge 0})$ and
$\wt\sH_0=\wt\sH_0(\{\Gamma_n\}_{n\ge 0})$ we replace
$\sD_{\Gamma_n}$ with $\{0\}$ (respect., $\sD_{\Gamma^*_n}$ with
$\{0\}$) for $n\ge m$, and $\sD_{\Gamma^*_n}$
 by $\sD_{\Gamma^*_m}$ (respect.,
$\sD_{\Gamma_n}$ by $\sD_{\Gamma_m}$) for $n> m$.
\item
the corresponding unitary elementary rotation takes the row
(respect., the column) form, i.e,
\[
\begin{array}{l}
{\bf
J}^{(r)}_{\Gamma_0}=\begin{bmatrix}\Gamma_0&I_{\sD_{\Gamma^*_0}}
\end{bmatrix}:\begin{array}{l}\sM\\\oplus\\\sD_{\Gamma^*_{0}}\end{array}\to\sN\;
\left (\mbox{respect},\;{\bf
J}^{(c)}_{\Gamma_0}=\begin{bmatrix}\Gamma_0\cr D_{\Gamma_0}
\end{bmatrix}:\sM\to
\begin{array}{l}\sN\\\oplus\\\sD_{\Gamma_{0}}\end{array}\right),\\
\\
{\bf
J}^{(r)}_{\Gamma_m}=\begin{bmatrix}\Gamma_m&I_{\sD_{\Gamma^*_m}}
\end{bmatrix}:\begin{array}{l}\sD_{\Gamma_{m-1}}\\\oplus\\\sD_{\Gamma^*_{m}}\end{array}
\to\sD_{\Gamma^*_{m-1}}\\
\qquad\qquad \left(\mbox{respect.},\; {\bf
J}^{(c)}_{\Gamma_m}=\begin{bmatrix}\Gamma_m\cr D_{\Gamma_m}
\end{bmatrix}:\sD_{\Gamma_{m-1}}\to
\begin{array}{l}\sD_{\Gamma^*_{m-1}}\\\oplus\\\sD_{\Gamma_{m}}\end{array}\right),\; m\ge 1.
\end{array}
\]
\end{enumerate}
Therefore, in definitions \eqref{OLM} of the block diagonal operator
matrices
\[
\cL_0=\cL_0(\{\Gamma_n\}_{n\ge 0}),\; \cM_0=\cM_0(\{\Gamma_n\}_{n\ge
0}),\;\mbox{and}\; \wt\cM_0=\wt\cM_0(\{\Gamma_n\}_{n\ge 0})
\]
we will replace
\begin{itemize}
\item ${\bf J}_{\Gamma_m}$ by ${\bf J}^{(r)}_{\Gamma_m}$
 and ${\bf J}_{\Gamma_n}$ by
$I_{\sD_{\Gamma^*_m}}$ for $n>m$, when $\Gamma_m$ is isometry,
\item ${\bf J}_{\Gamma_m}$ by ${\bf J}^{(c)}_{\Gamma_m}$,
 and ${\bf J}_{\Gamma_n}$ by
$I_{\sD_{\Gamma_m}}$ for $n>m$, when $\Gamma_m$ is co-isometry,
\item
${\bf J}_{\Gamma_m}$ by $\Gamma_m$, when $\Gamma_m$ is unitary.
\end{itemize}
In all these cases the block operators CMV matrices
$\cU_0=\cU_0(\{\Gamma_n\}_{n\ge 0})$ and
$\wt\cU_0=\wt\cU_0(\{\Gamma_n\}_{n\ge 0})$ are defined by means the
products
$\cU_0=\cL_0\cM_0,\;\wt\cU_0=\wt\cM_0\cL_0.$
These matrices are five block-diagonal. In the case when the
operator $\Gamma_m$ is unitary the block operator CMV matrices
$\cU_0$ and $\wt\cU_0$ are finite and otherwise they are
semi-infinite.

As before the truncated block operator CMV matrices
$\cT_0=\cT_0((\{\Gamma_n\}_{n\ge 0})$ and
$\wt\cT_0=\wt\cT_0(\{\Gamma_n\}_{n\ge 0})$ are defined by
\eqref{TRUNC} and \eqref{TRUNCT}, i.e.,
$$\cT_0=P_{\sH_0}\cU_0\uphar\sH_0,\;\wt\cT_0=P_{\wt\sH_0}\wt\cU_0\uphar\wt\sH_0.$$
The operators $\cT_0$ and $\wt\cT_0$ are unitarily equivalent
completely non-unitary contractions and Proposition \ref{COIUN} hold
true. The operators given by truncated block operator CMV matrices
$\cT_m$ and $\wt\cT_m$ obtaining from $\cU_0$ and $\wt\cU_0$ by
deleting first $m+1$ rows and $m+1$ columns are
\begin{itemize}
\item
co-shifts of the form
\[
\cT_m=\wt\cT_m=\begin{bmatrix}0&I_{\sD_{\Gamma^*_m}}&0&0&\ldots\cr
0&0&I_{\sD_{\Gamma^*_m}}&0&\ldots\cr
0&0&0&I_{\sD_{\Gamma^*_m}}&\ldots\cr
\vdots&\vdots&\vdots&\vdots&\vdots
\end{bmatrix}:\begin{array}{l}\sD_{\Gamma^*_m}\\\oplus\\\sD_{\Gamma^*_m}\\\oplus\\\vdots
\end{array}\to \begin{array}{l}\sD_{\Gamma^*_m}\\\oplus\\\sD_{\Gamma^*_m}\\\oplus\\\vdots
\end{array},
\]
when $\Gamma_m$ is isometry,
\item
 the unilateral shifts of the form
\[
\cT_m=\wt\cT_m=\begin{bmatrix}0&0&0&0&\ldots\cr
I_{\sD_{\Gamma_m}}&0&0&0&\ldots\cr
0&I_{\sD_{\Gamma_m}}&0&0&\ldots\cr
\vdots&\vdots&\vdots&\vdots&\vdots
\end{bmatrix}:\begin{array}{l}\sD_{\Gamma_m}\\\oplus\\\sD_{\Gamma_m}\\\oplus\\\vdots
\end{array}\to \begin{array}{l}\sD_{\Gamma_m}\\\oplus\\\sD_{\Gamma_m}\\\oplus\\\vdots
\end{array},
\]
 when $\Gamma_m$ is co-isometry.
 \end{itemize}
One can see that Proposition \ref{COIUN} remains true.

The conservative systems
\[
\zeta_0=\{\cU_0;\sM,\sN,\sH_0\},\;\wt\zeta_0=\{\wt\cU_0;\sM,\sN,\wt\sH_0\}.
\]
are simple and unitarily equivalent and, moreover, Theorem
\ref{cmvmod} remains valid.

In order to obtain precise forms of $\cU_0$, $\wt\cU_0$, $\cT_0$,
and $\wt\cT_0$ one can consider the following cases:
\begin{enumerate}
\item $\Gamma_{2N}$ is isometric (co-isometric) for some $N$,
\item $\Gamma_{2N+1}$ is isometric (co-isometric) for some
$N$,
\item the operator $\Gamma_{2N}$ is unitary for some $N$,
\item the operator $\Gamma_{2N+1}$ is unitary for some $N$.
\end{enumerate}
In the following we consider all these situations and will give the
forms of truncated CMV matrices. We use the sub-matrices  defined by
\eqref{submtr1} and \eqref{submtr2}.

\subsection{$\Gamma_{2N}$ is isometric}
Define
\[
\begin{array}{l}
\sH_0=\wt\sH_0=\sD_{\Gamma^*_0}\bigoplus
\sD_{\Gamma^*_0}\bigoplus\ldots,\;
\mbox{if}\;N=0,\\
\sH_0=\left(
\bigoplus\limits_{n=0}^{N-1}\begin{array}{l}\sD_{\Gamma_{2n}}\\\oplus\\\sD_{\Gamma^*_{2n+1}}\end{array}\right)
\bigoplus
\sD_{\Gamma^*_{2N}}\bigoplus\sD_{\Gamma^*_{2N}}\bigoplus\ldots\bigoplus\sD_{\Gamma^*_{2N}}\bigoplus\ldots,\\
\wt\sH_0=\left(
\bigoplus\limits_{n=0}^{N-1}\begin{array}{l}\sD_{\Gamma^*_{2n}}\\\oplus\\\sD_{\Gamma_{2n+1}}\end{array}\right)
\bigoplus
\sD_{\Gamma^*_{2N}}\bigoplus\sD_{\Gamma^*_{2N}}\bigoplus\ldots\bigoplus\sD_{\Gamma^*_{2N}}\bigoplus\ldots,
\; N\ge 1.
\end{array}
\]
 Define the unitary operators
\[
\begin{array}{l}
\cM_0=I_{\sM\oplus\sH_0},\;\wt\cM_0=I_{\sN\oplus\sH_0},\;
N=0,\\
\cM_0=I_\sM\bigoplus\left(\bigoplus\limits_{n=1}^{N}{\bf
J}_{\Gamma_{2n-1}}\right)\bigoplus I_{\sD_{\Gamma^*_{2N}}}\bigoplus
I_{\sD_{\Gamma^*_{2N}}}\bigoplus\ldots
:\sM\bigoplus\sH_0\to \sM\bigoplus\wt\sH_0\\
\wt \cM_0=I_\sN\bigoplus\left(\bigoplus\limits_{n=1}^{N}{\bf
J}_{\Gamma_{2n-1}}\right)\bigoplus I_{\sD_{\Gamma^*_{2N}}}\bigoplus
I_{\sD_{\Gamma^*_{2N}}}\bigoplus\ldots:\sN\bigoplus\sH_0\to\sN\bigoplus\wt\sH_0,\;
N\ge 1.
\end{array}
\]
The unitary  operator
$\cL_0:\sM\bigoplus\wt\sH_0\to\sN\bigoplus\sH_0$ is defined as
follows
\[
\begin{array}{l}
\cL_0=\left\{\begin{array}{l} {\bf J}^{(r)}_{\Gamma_0}\bigoplus
I_{\sD_{\Gamma^*_0}}\bigoplus I_{\sD_{\Gamma^*_0}}\bigoplus\ldots,\; N=0,\\
{\bf J}_{\Gamma_0}\bigoplus{\bf
J}^{(r)}_{\Gamma_2}
\bigoplus I_{\sD_{\Gamma^*_{2}}}\bigoplus
I_{\sD_{\Gamma^*_{2}}}\bigoplus\ldots,\;\mbox{if}\; N=1,\\
{\bf J}_{\Gamma_0}\bigoplus\left(\bigoplus\limits_{n=1}^{N-1}{\bf
J}_{\Gamma_{2n}}\right)\bigoplus {\bf J}^{(r)}_{\Gamma_{2N}}
\bigoplus I_{\sD_{\Gamma^*_{2N}}}\bigoplus
I_{\sD_{\Gamma^*_{2N}}}\bigoplus\ldots,\;\mbox{if}\; N\ge 2
\end{array}\right..
\end{array}
\]
Define
$\cU_0=\cL_0\cM_0, \wt\cU_0=\wt\cM_0\cL_0.$ In particular, if the
operator $\Gamma_0$ is isometric, then
\[
\cU_0=\wt\cU_0=\begin{bmatrix}\Gamma_0&I_{\sD_{\Gamma^*_0}}&0&0&0&0&\ldots\cr
0&0&I_{\sD_{\Gamma^*_0}}&0&0&0&\ldots\cr
0&0&0&I_{\sD_{\Gamma^*_0}}&0&0&\ldots\cr
0&0&0&0&I_{\sD_{\Gamma^*_0}}&0&\ldots\cr
\vdots&\vdots&\vdots&\vdots&\vdots&\vdots&\vdots
\end{bmatrix}.
\]
The block operator truncted CMV matrices $\cT_0$ and $\wt \cT_0$ are
the products (for $N\ge 1$):
\[
\begin{array}{l}
\cT_0=\left({-\Gamma^*_0}\oplus\left(\bigoplus\limits_{n=1}^{N-1}{\bf
J}_{\Gamma_{2n}}\right)\oplus {\bf J}^{(r)}_{\Gamma_{2N}} \oplus
I_{\sD_{\Gamma^*_{2N}}}\oplus
I_{\sD_{\Gamma^*_{2N}}}\oplus\ldots\right)\\
\qquad\times\left(\left(\bigoplus\limits_{n=1}^{N}{\bf
J}_{\Gamma_{2n-1}}\right)\oplus I_{\sD_{\Gamma^*_{2N}}}\oplus
I_{\sD_{\Gamma^*_{2N}}}\oplus\ldots \right),\\
\wt\cT_0=\left(\left(\bigoplus\limits_{n=1}^{N}{\bf
J}_{\Gamma_{2n-1}}\right)\oplus I_{\sD_{\Gamma^*_{2N}}}\oplus
I_{\sD_{\Gamma^*_{2N}}}\oplus\ldots \right)\\
\qquad\times
\left({-\Gamma^*_0}\oplus\left(\bigoplus\limits_{n=1}^{N-1}{\bf
J}_{\Gamma_{2n}}\right)\oplus {\bf J}^{(r)}_{\Gamma_{2N}} \oplus
I_{\sD_{\Gamma^*_{2N}}}\oplus
I_{\sD_{\Gamma^*_{2N}}}\oplus\ldots\right).
\end{array}
\]
Calculations give {\footnotesize
\[
\cT_0=\left[\begin{array}{c|c} {\cS_{N-1}}&
\begin{array}{ccc}
0&0&\ldots\cr \vdots&\vdots&\vdots\cr 0&0&\ldots\cr
I_{\sD_{\Gamma^*_{2N}}}&0&\ldots
\end{array}\\
\hline
{\bf 0}&
\begin{array}{ccccc}
0&I_{\sD_{\Gamma^*_{2N}}}&0&0&\ldots\cr
0&0&I_{\sD_{\Gamma^*_{2N}}}&0&\ldots\cr
\vdots&\vdots&\vdots&\vdots&\vdots
\end{array}
\end{array}
\right],\; 
\wt \cT_0=\left[\begin{array}{c|c} {\wt \cS_{N-1}}&
\begin{array}{ccc}
0&0&\ldots\cr  \vdots&\vdots&\vdots\cr 0&0&\ldots \cr
D_{\Gamma^*_{2N-1}}&0&\ldots\cr -\Gamma^*_{2N-1}&0&\ldots
\end{array}\\
\hline
{\bf 0}&
\begin{array}{ccccc}
0&I_{\sD_{\Gamma^*_{2N}}}&0&0&\ldots\cr
0&0&I_{\sD_{\Gamma^*_{2N}}}&0&\ldots\cr
\vdots&\vdots&\vdots&\vdots&\vdots
\end{array}
\end{array}
\right].
\]}

\subsection{$\Gamma_{2N}$ is co-isometric}  Then $\Gamma_{n}=0$, $\sD_{\Gamma_{n}}=\sD_{\Gamma_{2N}}$,
$D_{\Gamma_{n}}=I_{\sD_{\Gamma_{2N}}}$ for $n> 2N$. Define
\[
\begin{array}{l}
\sH_0=\wt\sH_0=\bigoplus\limits_{n=0}^\infty\sD_{\Gamma_0},\;
\mbox{if}\;N=0,\\
\sH_0=\left(
\bigoplus\limits_{n=0}^{N-1}\begin{array}{l}\sD_{\Gamma_{2n}}\\\oplus\\\sD_{\Gamma^*_{2n+1}}\end{array}\right)\bigoplus
\sD_{\Gamma_{2N}}\bigoplus\sD_{\Gamma_{2N}}\bigoplus\ldots\bigoplus\sD_{\Gamma_{2N}}\bigoplus\ldots,\\
 \wt\sH_0=\left(
\bigoplus\limits_{n=0}^{N-1}\begin{array}{l}\sD_{\Gamma^*_{2n}}\\\oplus\\\sD_{\Gamma_{2n+1}}\end{array}\right)\bigoplus
\sD_{\Gamma_{2N}}\bigoplus\sD_{\Gamma_{2N}}\bigoplus\ldots\bigoplus\sD_{\Gamma_{2N}}\bigoplus\ldots,\;\mbox{if}\;
N\ge 1.
\end{array}
\]
Define the unitary operators
$\cL_0:\sM\bigoplus\wt\sH_0\to\sN\bigoplus\sH_0$,
$\cM_0:\sM\bigoplus\sH_0\to \sM\bigoplus\wt\sH_0$, and
$\wt\cM_0:\sN\bigoplus\sH_0\to\sN\bigoplus\wt\sH_0$ as follows
\[
\begin{array}{l}
\cL_0={\bf J}^{(c)}_{\Gamma_0}\bigoplus I_{\sD_{\Gamma_0}}\bigoplus
I_{\sD_{\Gamma_0}}\bigoplus\ldots\;\mbox{if}\;N=0,\\
\cL_0={\bf
J}_{\Gamma_0}\bigoplus\left(\bigoplus\limits_{n=0}^{N-1}{\bf
J}_{\Gamma_{2n}}\right)\bigoplus{\bf J}^{(c)}_{\Gamma_{
2N}}
\bigoplus I_{\sD_{\Gamma_{2N}}}\bigoplus
I_{\sD_{\Gamma_{2N}}}\bigoplus\ldots,\;\mbox{if}\; N\ge 1,\\
\cM_0=I_\sM\bigoplus I_{\sD_{\Gamma_0}}\bigoplus
I_{\sD_{\Gamma_0}}\bigoplus\ldots\;\mbox{if}\;N=0,\\
\cM_0=I_\sM\bigoplus\left(\bigoplus\limits_{n=1}^{N}{\bf
J}_{\Gamma_{2n-1}}\right)\bigoplus I_{\sD_{\Gamma_{2N}}}\bigoplus
I_{\sD_{\Gamma_{2N}}}\bigoplus\ldots,\;\mbox{if}\; N\ge 1,
\\
\wt\cM_0=I_\sN\bigoplus I_{\sD_{\Gamma_0}}\bigoplus
I_{\sD_{\Gamma_0}}\bigoplus\ldots\;\mbox{if}\;N=0,\\
 \wt
\cM_0=I_\sN\bigoplus\left(\bigoplus\limits_{n=1}^{N}{\bf
J}_{\Gamma_{2n-1}}\right)\bigoplus I_{\sD_{\Gamma_{2N}}}\bigoplus
I_{\sD_{\Gamma_{2N}}}\bigoplus\ldots,\;\mbox{if}\; N\ge 1.
\end{array}
\]
Finally define $\cU_0=\cL_0\cM_0$ and $\wt \cU_0=\wt \cM_0\cL_0$.

In particular, if the operator $\Gamma_{0}$ is co-isometric, then
$\sH_0=\wt\sH_0=\sD_{\Gamma_0}\bigoplus\sD_{\Gamma_0}\bigoplus\ldots,
$
\[
\cM_0=I_{\sM\oplus\sH_0},\; \wt\cM_0=I_{\sN\oplus\sH_0},\;
\cL_0={\bf J}^{(c)}_{\Gamma_0}\bigoplus I_{\sD_{\Gamma_0}}\bigoplus
I_{\sD_{\Gamma_0}}\bigoplus\ldots,
\]
\[
\cU_0=\wt\cU_0=\begin{bmatrix}\Gamma_0&0&0&0&\ldots\cr
D_{\Gamma_0}&0&0&0&\ldots\cr
 0&I_{\sD_{\Gamma_0}}&0&0&\ldots\cr
0&0&I_{\sD_{\Gamma_0}}&0&\ldots\cr
\vdots&\vdots&\vdots&\vdots&\vdots
  \end{bmatrix}.
\]
If $N\ge 1$, then truncated CMV matrices $\cT_0$ and $\wt\cT_0$ are
of the form {\footnotesize
\[
\cT_0=\left[\begin{array}{c|c} {\cS_{N-1}}&{\bf 0}\\
\hline
\begin{array}{ccccc}
0&\ldots&0&D_{\Gamma_{2N}}D_{\Gamma_{2N-1}}&-D_{\Gamma_{2N}}\Gamma^*_{2N-1}
\cr 0&\ldots&0&0&0\cr \vdots&\vdots&\vdots&\vdots&\vdots
\end{array}&
\begin{array}{cccc}
0&0&0&\ldots\cr I_{\sD_{\Gamma_{2N}}}&0&0&\ldots\cr
0&I_{\sD_{\Gamma_{2N}}}&0&\ldots\cr \vdots&\vdots&\vdots&\vdots
\end{array}
\end{array}
\right],
\]
\[
\wt\cT_0=\left[\begin{array}{c|c} {\wt \cS_{N-1}}&{\bf 0}\\
\hline
\begin{array}{ccccc}
0&0&\ldots&0&D_{\Gamma_{2N}} \cr 0&0&\ldots&0&0\cr
\vdots&\vdots&\vdots&\vdots&\vdots
\end{array}&
\begin{array}{cccc}
0&0&0&\ldots\cr I_{\sD_{\Gamma_{2N}}}&0&0&\ldots\cr
0&I_{\sD_{\Gamma_{2N}}}&0&\ldots\cr \vdots&\vdots&\vdots&\vdots
\end{array}
\end{array}
\right].
\]
}

\subsection{$\Gamma_{2N+1}$ is isometric}
 In this case $\Gamma_{n}=0$,
$\sD_{\Gamma^*_{n}}=\sD_{\Gamma^*_{2N+1}}$,
$D_{\Gamma^*_{n}}=I_{\sD_{\Gamma^*_{2N+1}}}$ for $n> 2N+1$. Define
\[
\begin{array}{l}
\sH_0=\left(
\bigoplus\limits_{n=0}^{N}\begin{array}{l}\sD_{\Gamma_{2n}}\\\oplus\\\sD_{\Gamma^*_{2n+1}}\end{array}\right)\bigoplus
\sD_{\Gamma^*_{2N+1}}\bigoplus\sD_{\Gamma^*_{2N+1}}\bigoplus\ldots\bigoplus\sD_{\Gamma^*_{2N+1}}\bigoplus\ldots,\\
  \wt \sH_0=\sD_{\Gamma^*_0}\bigoplus\sD_{\Gamma^*_1}\bigoplus\sD_{\Gamma^*_1}\bigoplus\ldots
  \bigoplus\sD_{\Gamma^*_1}\bigoplus\ldots,\quad\mbox{if}\quad
N=0,\\
 \wt\sH_0=\left(
\bigoplus\limits_{n=0}^{N-1}\begin{array}{l}\sD_{\Gamma^*_{2n}}\\\oplus\\\sD_{\Gamma_{2n+1}}\end{array}\right)\bigoplus
\sD_{\Gamma^*_{2N}}\bigoplus\sD_{\Gamma^*_{2N+1}}\bigoplus\ldots\bigoplus\sD_{\Gamma^*_{2N+1}}\bigoplus\ldots,\quad\mbox{if}\quad
N\ge 1.
\end{array}
\]

 Define the unitary operators
\[
\begin{array}{l}
\cM_0=I_\sM\bigoplus {\bf
J}^{(r)}_{\Gamma_1}
\bigoplus I _{\sD_{\Gamma^*_1}}\bigoplus
I_{\sD_{\Gamma^*_1}}\bigoplus\ldots (N=0),\\
 \cM_0=I_\sM\bigoplus\left(\bigoplus\limits_{n=1}^{N}{\bf
J}_{\Gamma_{2n-1}}\right)\bigoplus{\bf
J}^{(r)}_{\Gamma_{2N+1}}
\bigoplus I_{\sD_{\Gamma^*_{2N+1}}}
\bigoplus I_{\sD_{\Gamma^*_{2N+1}}}\bigoplus\ldots (N\ge 1),\\
\cL_0={\bf J}_{\Gamma_0}\bigoplus I _{\sD_{\Gamma^*_1}}\bigoplus
I_{\sD_{\Gamma^*_1}}\bigoplus\ldots (N=0),\\
 \cL_0={\bf
J}_{\Gamma_0}\bigoplus\left(\bigoplus\limits_{n=1}^{N}{\bf
J}_{\Gamma_{2n}}\right)
\bigoplus I_{\sD_{\Gamma^*_{2N+1}}}\bigoplus I_{\sD_{\Gamma^*_{2N+1}}}
\bigoplus\ldots(N\ge 1),\\
\wt \cM_0=I_\sN\bigoplus{\bf J}^{(r)}_{\Gamma_1}
\bigoplus I _{\sD_{\Gamma^*_1}}\bigoplus
I_{\sD_{\Gamma^*_1}}\bigoplus\ldots (N=0),\\
\wt \cM_0=I_\sN\bigoplus\left(\bigoplus\limits_{n=1}^{N}{\bf
J}_{\Gamma_{2n-1}}\right)\bigoplus{\bf
J}^{(r)}_{\Gamma_{2N+1}}
\bigoplus I_{\sD_{\Gamma^*_{2N+1}}}\bigoplus
I_{\sD_{\Gamma^*_{2N+1}}}\bigoplus\ldots (N\ge 1).
\end{array}
\]
Define $\cU_0=\cL_0\cM_0$ and $\wt \cU_0=\wt \cM_0\cL_0$.

If the operator $\Gamma_1$ is isometric, then
\[
\begin{array}{l}
\sH_0=
\begin{array}{l}\sD_{\Gamma_{0}}\\\oplus\\\sD_{\Gamma^*_{1}}\end{array}\bigoplus
\sD_{\Gamma^*_{1}}\bigoplus\sD_{\Gamma^*_{1}}\bigoplus\ldots\bigoplus\sD_{\Gamma^*_{1}}\bigoplus\ldots,\\
  \wt \sH_0=\sD_{\Gamma^*_0}\bigoplus\sD_{\Gamma^*_1}\bigoplus\sD_{\Gamma^*_1}\bigoplus\ldots
  \bigoplus\sD_{\Gamma^*_1}\bigoplus\ldots,
\end{array}
\]
{\footnotesize
\[
\begin{array}{l}
\cU_0=\begin{bmatrix}\Gamma_0&D_{\Gamma^*_0}\Gamma_1&
D_{\Gamma^*_0}&0&0&0&0&0&\ldots\cr
D_{\Gamma_0}&-\Gamma^*_0\Gamma_1&-\Gamma^*_0&0&0&0&0&0&\ldots\cr
0&0&0&I_{\sD_{\Gamma^*_1}}&0&0&0&0&\ldots\cr
0&0&0&0&I_{\sD_{\Gamma^*_1}}&0&0&0&\ldots\cr
\vdots&\vdots&\vdots&\vdots&\vdots&\vdots&\vdots&\vdots&\vdots
\end{bmatrix},
\end{array}
\]
\[
\begin{array}{l}
\wt\cU_0=\begin{bmatrix}\Gamma_0&D_{\Gamma^*_0}&
0&0&0&0&0&0&\ldots\cr
\Gamma_1D_{\Gamma_0}&-\Gamma_1\Gamma^*_0&I_{\sD_{\Gamma^*_1}}&0&0&0&0&0&\ldots\cr
0&0&0&I_{\sD_{\Gamma^*_1}}&0&0&0&0&\ldots\cr
0&0&0&0&I_{\sD_{\Gamma^*_1}}&0&0&0&\ldots\cr
\vdots&\vdots&\vdots&\vdots&\vdots&\vdots&\vdots&\vdots&\vdots
\end{bmatrix}.
\end{array}
\]}
If $N\ge 1$, then truncated CMV matrices $\cT_0$ and $\wt\cT_0$ take
the form {\footnotesize
\[
\cT_0=\left[\begin{array}{c|c} {\cS_{N-1}}&\begin{array}{cccc}
0&0&0&\ldots\cr  \vdots&\vdots&\vdots&\vdots\cr 0&0&0&\ldots\cr
D_{\Gamma^*_{2N+1}}\Gamma_{2N+1}&D_{\Gamma^*_{2N}}&0&\ldots
\end{array}
\\
\hline
\begin{array}{cccccc}
0&\ldots&0&D_{\Gamma_{2N}}D_{\Gamma_{2N-1}}&-D_{\Gamma_{2N}}\Gamma^*_{2N-1}\cr
0&\ldots&0&0&0\cr \vdots&\vdots&\vdots&\vdots&\vdots
\end{array}&
\begin{array}{cccccc}
-\Gamma^*_{2N}\Gamma_{2N+1}&-\Gamma^*_{2N}&0&0&\ldots\cr
0&0&I_{\sD_{\Gamma^*_{2N+1}}}&0&\ldots\cr
0&0&0&I_{\sD_{\Gamma^*_{2N+1}}}&\ldots\cr
\vdots&\vdots&\vdots&\vdots&\vdots
\end{array}
\end{array}
\right],
\]
\[
\wt\cT_0=\left[\begin{array}{c|c} {\wt
\cS_{N-1}}&\begin{array}{cccc} 0&0&\ldots\cr
\vdots&\vdots&\vdots\cr 0&0&\ldots \cr
D_{\Gamma^*_{2N-1}}D_{\Gamma^*_{2N}}&0&\ldots\cr
-{\Gamma^*_{2N-1}}D_{\Gamma^*_{2N}}&0&\ldots
\end{array}
\\
\hline
\begin{array}{cccccc}
0&\ldots&0&\Gamma_{2N+1}D_{\Gamma_{2N}}\cr0&\ldots&0&0
\cr \vdots&\vdots&\vdots&\vdots
\end{array}&
\begin{array}{cccccc}
-\Gamma_{2N+1}\Gamma^*_{2N}&I_{\sD_{\Gamma^*_{2N+1}}}&0&0&0&\ldots\cr
0&0&I_{\sD_{\Gamma^*_{2N+1}}}&0&0&\ldots\cr
0&0&0&I_{\sD_{\Gamma^*_{2N+1}}}&0&\ldots\cr
\vdots&\vdots&\vdots&\vdots&\vdots&\vdots
\end{array}
\end{array}
\right].
\]
}

\subsection{$\Gamma_{2N+1}$ is co-isometric }  Then $\Gamma_{n}=0$, $\sD_{\Gamma_{n}}=\sD_{\Gamma_{2N+1}}$,
$D_{\Gamma_{n}}=I_{\sD_{\Gamma_{2N+1}}}$ for $n> 2N+1$. Define
\[
\begin{array}{l}
\sH_0=\sD_{\Gamma_0}\bigoplus\sD_{\Gamma_1}\bigoplus\sD_{\Gamma_1}\bigoplus\ldots,\;
\mbox{if}\;N=0,\\
\sH_0=\left(
\bigoplus\limits_{n=0}^{N-1}\begin{array}{l}\sD_{\Gamma_{2n}}\\\oplus\\\sD_{\Gamma^*_{2n+1}}\end{array}\right)\bigoplus
\sD_{\Gamma_{2N}}\bigoplus\sD_{\Gamma_{2N+1}}\bigoplus\ldots\bigoplus\sD_{\Gamma_{2N+1}}\bigoplus\ldots\;\mbox{if}\;
N\ge 1,
\end{array}
\]
\[
\begin{array}{l}
\wt\sH_0=
\sD_{\Gamma^*_0}\bigoplus\sD_{\Gamma_1}\bigoplus\sD_{\Gamma_1}\bigoplus\ldots,\;
\mbox{if}\;N=0, \\
 \wt\sH_0=\left(
\bigoplus\limits_{n=0}^{N}\begin{array}{l}\sD_{\Gamma^*_{2n}}\\\oplus\\\sD_{\Gamma_{2n+1}}\end{array}\right)\bigoplus
\sD_{\Gamma_{2N+1}}\bigoplus\sD_{\Gamma_{2N+1}}\bigoplus\ldots\bigoplus\sD_{\Gamma_{2N+1}}\bigoplus\ldots\;\mbox{if}\;
N\ge 1.
\end{array}
\]
Define operators
\[
\begin{array}{l}
\cL_0={\bf J}_{D_{\Gamma_0}}\bigoplus
I_{\sD_{\Gamma_1}}\bigoplus I_{\sD_{\Gamma_1}}\bigoplus\ldots\;\mbox{if}\;N=0,\\
 \cL_0= {\bf
J}_{\Gamma_0}\bigoplus\left(\bigoplus\limits_{n=1}^{N}{\bf
J}_{\Gamma_{2n}}\right) \bigoplus I_{\sD_{\Gamma_{2N+1}}}\bigoplus
I_{\sD_{\Gamma_{2N+1}}}\bigoplus\ldots\;\mbox{if}\; N\ge 1\;:\sM\bigoplus\wt\sH_0\to\sN\bigoplus\sH_0,\\
\end{array}
\]
\[
\begin{array}{l}
\cM_0=I_\sM\bigoplus{\bf J}^{(c)}_{\Gamma_1}\bigoplus
I_{\sD_{\Gamma_1}}\bigoplus I_{\sD_{\Gamma_1}}\bigoplus\ldots\;\mbox{if}\;N=0, \\
 \cM_0=I_\sM\bigoplus\left(\bigoplus\limits_{n=1}^{N}{\bf
J}_{\Gamma_{2n-1}}\right)\bigoplus{\bf J}^{(c)}_{\Gamma_{2N+1}}
\bigoplus I_{\sD_{\Gamma_{2N+1}}} \bigoplus
\ldots\;\mbox{if}\; N\ge 1\; :\sM\bigoplus\sH_0\to
\sM\bigoplus\wt\sH_0,
\end{array}
\]
\[
\begin{array}{l}
\wt\cM_0=I_\sN\bigoplus{\bf J}^{(c)}_{\Gamma_1}\bigoplus
I_{\sD_{\Gamma_1}}\bigoplus I_{\sD_{\Gamma_1}}\bigoplus\ldots\;\mbox{if}\;N=0, \\
\wt\cM_0=I_\sN\bigoplus\left(\bigoplus\limits_{n=1}^{N}{\bf
J}_{\Gamma_{2n-1}}\right)\bigoplus{\bf J}^{(c)}_{\Gamma_{2N+1}}
\bigoplus I_{\sD_{\Gamma_{2N+1}}}\bigoplus\ldots \;\mbox{if}\; N\ge
1\;:\sM\bigoplus\sH_0\to \sM\bigoplus\wt\sH_0.
\end{array}
\]
Define $\cU_0=\cL_0\cM_0$ and $\wt \cU_0=\wt \cM_0\cL_0$.

If $\Gamma_1$ is co-isometric, then {\footnotesize
\[
\cU_0=\begin{bmatrix}\Gamma_0&D_{\Gamma^*_0}\Gamma_1&0&0&0&\ldots\cr
D_{\Gamma_0}&-\Gamma^*_0\Gamma_1&0&0&0&\ldots\cr
 0&{D_{\Gamma_1}}&0&0&0&\ldots\cr
0&0&I_{\sD_{\Gamma_1}}&0&0&\ldots\cr
0&0&0&I_{\sD_{\Gamma_1}}&0&\ldots\cr
\vdots&\vdots&\vdots&\vdots&\vdots&\vdots
  \end{bmatrix},\;
\wt\cU_0=\begin{bmatrix}\Gamma_0&D_{\Gamma^*_0}&0&0&0&\ldots\cr
\Gamma_1D_{\Gamma_0}&-\Gamma_1\Gamma^*_0&0&0&0&\ldots\cr
 D_{\Gamma_1}D_{\Gamma_0}&-{D_{\Gamma_1}}\Gamma^*_0&0&0&0&\ldots\cr
0&0&I_{\sD_{\Gamma_1}}&0&0&\ldots\cr
0&0&0&I_{\sD_{\Gamma_1}}&0&\ldots\cr
\vdots&\vdots&\vdots&\vdots&\vdots&\vdots
  \end{bmatrix}.
  \]
  }
If $N\ge 1$, then truncated CMV matrices $\cT_0$ and $\wt\cT_0$ in
this case take the form {\footnotesize

\[
\cT_0=\left[\begin{array}{c|c} {\cS_{N-1}}&\begin{array}{cccccc}
0&0&\ldots\cr  \vdots&\vdots&\vdots\cr 0&0&\ldots\cr
D_{\Gamma^*_{2N}}\Gamma_{2N+1}&0&\ldots
\end{array}
\\
\hline
\begin{array}{ccccc}
0&\ldots&0&D_{\Gamma_{2N}}D_{\Gamma_{2N-1}}&-D_{\Gamma_{2N}}\Gamma^*_{2N-1}\cr
0&\ldots&0&0&0\cr \vdots&\vdots&\vdots&\vdots&\vdots
\end{array}&
\begin{array}{cccccc}
-\Gamma^*_{2N}\Gamma_{2N+1}&0&0&\ldots\cr
D_{\Gamma_{2N+1}}&0&0&\ldots\cr
0&I_{\sD_{\Gamma^*_{2N+1}}}&0&\ldots\cr
\vdots&\vdots&\vdots&\vdots&\vdots
\end{array}
\end{array}
\right],
\]

\[
\wt \cT_0=\left[\begin{array}{c|c} {\wt
\cS_{N-1}}&\begin{array}{cccccc} 0&0&\ldots\cr
\vdots&\vdots&\vdots\cr 0&0&\ldots \cr
D_{\Gamma^*_{2N-1}}D_{\Gamma^*_{2N}}&0&\ldots\cr
-{\Gamma^*_{2N-1}}D_{\Gamma^*_{2N}}&0&\ldots
\end{array}
\\
\hline
\begin{array}{ccccc}
0&\ldots&0&\Gamma_{2N+1}D_{\Gamma_{2N}}\cr
0&\ldots&0&D_{\Gamma_{2N+1}}D_{\Gamma_{2N}}\cr 0&\ldots&0&0\cr
\vdots&\vdots&\vdots&\vdots
\end{array}&
\begin{array}{cccccc}
-\Gamma_{2N+1}\Gamma^*_{2N}&0&0&\ldots\cr
-D_{\Gamma_{2N+1}}\Gamma^*_{2N}&0&0&\ldots\cr
0&I_{\sD_{\Gamma^*_{2N+1}}}&0&0&\ldots\cr
0&0&I_{\sD_{\Gamma^*_{2N+1}}}&0&\ldots\cr
\vdots&\vdots&\vdots&\vdots&\vdots
\end{array}
\end{array}
\right].
\]
}

\subsection{$\Gamma_{2N}$ is unitary} In this case
\[
\sH_0=
\bigoplus\limits_{n=0}^{N-1}\begin{array}{l}\sD_{\Gamma_{2n}}\\\oplus\\\sD_{\Gamma^*_{2n+1}}\end{array},\;
\wt\sH_0=
\bigoplus\limits_{n=0}^{N-1}\begin{array}{l}\sD_{\Gamma^*_{2n}}\\\oplus\\\sD_{\Gamma_{2n+1}}\end{array},
\]
\[
\cU_0=\left(\bJ_{\Gamma_0}\oplus\bJ_{\Gamma_2}\oplus\cdots\oplus
\bJ_{\Gamma_{2(N-1)}}\oplus\Gamma_{2N}\right)\times\left(I_{\sM}\oplus
\bJ_{\Gamma_1}\oplus\cdots\oplus\bJ_{\Gamma_{2N-1}}\right),
\]
\[
\wt \cU_0=\left(I_{\sN}\oplus
\bJ_{\Gamma_1}\oplus\cdots\oplus\bJ_{\Gamma_{2N-1}}\right)\times
\left(\bJ_{\Gamma_0}\oplus\bJ_{\Gamma_2}\oplus\cdots\oplus
\bJ_{\Gamma_{2(N-1)}}\oplus\Gamma_{2N}\right).
\]
 If $N\ge 1$, then
$$\cT_0= \cS_{N-1},\; \wt\cT_0= \wt \cS_{N-1}.$$

\subsection {$\Gamma_{2N+1}$ is unitary} Then
\[
\begin{array}{l}
\sH_0=\sD_{\Gamma_0},\;\wt\sH_0=\sD_{\Gamma^*_0}\;\mbox{if}\;N=0,\\
 \sH_0=
\bigoplus\limits_{n=0}^{N-1}\begin{array}{l}\sD_{\Gamma_{2n}}\\\oplus\\\sD_{\Gamma^*_{2n+1}}\end{array}\bigoplus\sD_{\Gamma_{2N}},\;
\wt\sH_0=
\bigoplus\limits_{n=0}^{N-1}\begin{array}{l}\sD_{\Gamma^*_{2n}}\\\oplus\\\sD_{\Gamma_{2n+1}}\end{array}\bigoplus\sD_{\Gamma^*_{2N}}\;\mbox{if}\;
N\ge 1
\end{array},
\]
\[
\cU_0=\left(\bJ_{\Gamma_0}\oplus\bJ_{\Gamma_2}\oplus\cdots\oplus
\bJ_{\Gamma_{2N}}\right)\times\left(I_{\sM}\oplus
\bJ_{\Gamma_1}\oplus\cdots\oplus\bJ_{\Gamma_{2N-1}}\oplus\Gamma_{2N+1}\right),
\]
\[
\wt \cU_0=\left(I_{\sN}\oplus
\bJ_{\Gamma_1}\oplus\cdots\oplus\bJ_{\Gamma_{2N-1}}\oplus\Gamma_{2N+1}\right)\times
\left(\bJ_{\Gamma_0}\oplus\bJ_{\Gamma_2}\oplus\cdots\oplus
\bJ_{\Gamma_{2N}}\right), \; N\ge 1.
\]
\[
\begin{array}{l}
\cU_0=\begin{bmatrix}\Gamma_0&D_{\Gamma^*_0}\cr
D_{\Gamma_0}&-\Gamma^*_0\end{bmatrix}\begin{bmatrix}I_\sM&0\cr
0&\Gamma_1
\end{bmatrix}=\begin{bmatrix}\Gamma_0&D_{\Gamma^*_0}\Gamma_1\cr D_{\Gamma_0}&-\Gamma^*_0\Gamma_1
\end{bmatrix},\\
\wt\cU_0=\begin{bmatrix}I_\sN&0\cr 0&\Gamma_1
\end{bmatrix}\begin{bmatrix}\Gamma_0&D_{\Gamma^*_0}\cr
D_{\Gamma_0}&-\Gamma^*_0\end{bmatrix}=\begin{bmatrix}\Gamma_0&D_{\Gamma^*_0}\cr
\Gamma_1D_{\Gamma_0}&-\Gamma_1\Gamma^*_0
\end{bmatrix} ,\;\mbox{if}\; N=0,
\end{array}
\]
 In this case if $N\ge 1$, then {\footnotesize
\[
\cT_0=\left[\begin{array}{c|c} \cS_{N-1} &
\begin{array}{cc}
0
\cr\vdots\cr 0\cr D_{\Gamma^*_{2N}}\Gamma_{2N+1}
\end{array}\\
\hline
\begin{array}{ccccc}
0&\ldots&0&D_{\Gamma_{2N}}D_{\Gamma_{2N-1}}&-D_{\Gamma_{2N}}{\Gamma^*_{2N-1}}
\end{array}&
\begin{array}{cc}
-\Gamma^*_{2N}\Gamma^*_{2N+1}
\end{array}
\end{array}
\right],
\]
\[
\wt\cT_0=\left[\begin{array}{c|c} \wt \cS_{N-1} &
\begin{array}{cc}
0
\cr\vdots\cr 0\cr D_{\Gamma^*_{2N-1}}D_{\Gamma^*_{2N}}\cr
-\Gamma_{2N-1}D_{\Gamma^*_{2N}}
\end{array}\\
\hline
\begin{array}{ccccc}
0&\ldots&0&{\Gamma_{2N+1}}D_{\Gamma_{2N}}
\end{array}&
\begin{array}{cc}
-\Gamma_{2N+1}\Gamma^*_{2N}
\end{array}
\end{array}
\right].
\]
}
 In particular if $N=1$ ($\Gamma_3$ is unitary), then {\footnotesize
\[
\cU_0=\begin{bmatrix}\Gamma_0&D_{\Gamma^*_0}\Gamma_1&
D_{\Gamma^*_0}D_{\Gamma^*_1}&0\cr
D_{\Gamma_0}&-\Gamma^*_0\Gamma_1&-\Gamma^*_0D_{\Gamma^*_1}&0\cr
0&\Gamma_2
D_{\Gamma_1}&-\Gamma_2\Gamma^*_1&D_{\Gamma^*_2}\Gamma_3\cr
0&D_{\Gamma_2}D_{\Gamma_1}&-D_{\Gamma_2}\Gamma^*_1&-\Gamma^*_2\Gamma_3
\end{bmatrix},\;\wt\cU_0=\begin{bmatrix}\Gamma_0&D_{\Gamma^*_0}&0&0&\cr
\Gamma_1D_{\Gamma_0}
&-\Gamma_1\Gamma^*_0&D_{\Gamma^*_1}\Gamma_2&D_{\Gamma^*_1}D_{\Gamma^*_2}\cr
D_{\Gamma_1}D_{\Gamma_0}&-D_{\Gamma_1}\Gamma^*_0&-\Gamma^*_1\Gamma_2&-\Gamma^*_1D_{\Gamma^*_2}\cr
0&0&\Gamma_3 D_{\Gamma_2}&-\Gamma_3\Gamma^*_2
\end{bmatrix}.
\]
}

\end{document}